\documentclass[11pt]{amsart}
\usepackage{fullpage}
\usepackage{amsfonts,amssymb,amsmath,amsthm}
\usepackage{hyperref}

\theoremstyle{plain}
\newtheorem{theorem}{Theorem}[section]
 \newtheorem{corollary}[theorem]{Corollary}
 \newtheorem{lemma}[theorem]{Lemma}
 \newtheorem{proposition}[theorem]{Proposition}
 \theoremstyle{definition}
 \newtheorem{definition}[theorem]{Definition}
 
 \theoremstyle{remark}
 \newtheorem{remark}[theorem]{Remark}
 \numberwithin{equation}{section}

\usepackage{mathrsfs}
\def \sL{\mathscr L}
\def \sF {\mathscr F}

\def \bC {\mathbb C}

\def \bN {\mathbb N}

\def \bR {\mathbb R}

\def \bZ {\mathbb Z}

\def \cA {\mathcal A}
\def \cB {\mathcal B}
\def \cC {\mathcal C}
\def \cD {\mathcal D}

\def \cF {\mathcal F}

\def \cH {\mathcal H}

\def \cL {\mathcal L}

\def \cR {\mathcal R}
\def \cS {\mathcal S}
\def \cR {\mathcal R}
\def \cR {\mathcal R}

\def \fg {\mathfrak g}

\def\eps{\varepsilon}
\def\R{{\mathbb R}}
\def\C{{\mathbb C}}
\def\N{{\mathbb N}}
\mathchardef\mhyphen="2D 

\def \Ghat{\widehat{G}}

\def \tr {{\rm tr}}
\def \Op  {{\rm Op}}
\def \id {{\rm id}}

\numberwithin{equation}{section}

\begin{document}
\title[]{Wick-Quantization on Groups and application to G\aa rding inequalities}
\author[L. Benedetto]{Lino Benedetto}
\address[L. Benedetto]{DMA, École normale supérieure, Université PSL, CNRS, 75005 Paris, France \& Univ Angers, CNRS, LAREMA, SFR MATHSTIC, F-49000 Angers, France} 
\email{lbenedetto@dma.ens.fr}
\author[C. Fermanian]{Clotilde~Fermanian~Kammerer}
\address[C. Fermanian Kammerer]{
Univ Paris Est Creteil, Univ Gustave Eiffel, CNRS, LAMA UMR8050, F-94010 Creteil, France \& Univ Angers, CNRS, LAREMA, SFR MATHSTIC, F-49000 Angers, France
}
\email{clotilde.fermanian@u-pec.fr}
\author[V. Fischer]{V\'eronique Fischer}\address[V. Fischer]
{University of Bath, Department of Mathematical Sciences, Bath, BA2 7AY, UK} 
\email{v.c.m.fischer@bath.ac.uk}

\begin{abstract} 
In this paper, we introduce Wick's quantization on groups and discuss its links with Kohn-Nirenberg's. By  quantization, we mean     an operation that associates an operator to a symbol. 
The notion of symbols for both quantizations  is based on representation theory via the group Fourier transform  and the  Plancherel theorem. 
As an application, we  
prove G\aa rding inequalities for three global symbolic pseudodifferential calculi on groups. 
 \end{abstract}


\maketitle

\makeatletter
\renewcommand\l@subsection{\@tocline{2}{0pt}{3pc}{5pc}{}}
\makeatother

\tableofcontents

\section{Introduction}

In this  paper, we show how,
as for the Kohn-Nirenberg quantization, the definition of the  Wick quantization extends naturally to groups
that  satisfy some hypotheses allowing for the definition of the group Fourier transform (based on representation theory), and   the associated Plancherel theorem. 
As a straightforward counterpart, we obtain the analogue of the Bargmann transform~\cite{Lerner} in the Euclidean case 
and a natural frame on graded Lie groups, based on the wave packets constructed in~\cite{FF0,FF1,FL}. 
This frame is different from the wavelets frame defined on stratified Lie groups by~\cite{lemarie}, see also more generally \cite{Fuhr}, though close in spirit; it is the analogue of the Gaussian frame used for constructing approximation of the Schr\"odinger propagator in the semi-classical limit~\cite{swart_rousse,robert}, with applications in the numerical analysis of quantum dynamics~\cite{LS}. We think that our construction of the Wick quantization, though quite simple, opens the way to various applications.

\medskip 

Here, as an application,
we  prove  G\aa rding inequalities for two global symbolic pseudodifferential 
calculi on compact and graded nilpotent Lie groups, discussing also the semi-classical calculus in the non-compact case. 
This topic, that is, 
G\aa rding inequalities for global pseudodifferential calculi on groups, has been the subject of many papers in recent years, see e.g. \cite{FRCras,RTGarding,CDR,CFR}. 

It turns out that 
on $\bR^n$, the links between the Kohn-Nirenberg and Wick quantizations provide 
some G\aa rding inequalities; this is briefly sketched in  Appendix \ref{app} for the H\"ormander calculus on~$\bR^n$
while a reference for the semi-classical is for instance Jean-Marc Bouclet's lecture notes \cite{bouclet}, see also~\cite{Zwobook,Lerner}. Though weaker than what is usually meant by `sharp G\aa rding inequality', these inequalities are interesting by themselves for applications and are still strong G\aa rding inequalities, in the sense that there is a gain of half a derivative or of a power of the semi-classical parameter if any.

We extend this approach to the case of groups: we prove
the G\aa rding inequalities that are summarised in the three following theorems,
although their statements will use the notation for the settings and the calculi recalled later on in the paper.   
The first inequality is set on compact Lie groups and considers the global symbolic pseudodifferential calculus proposed in \cite{RT,RTW},
studied in \cite{FJFA2015} and  briefly recalled in Section \ref{subsec_pseudoC_compact}. We will prove the $(\rho,\delta)$-generalisation of the following G\aa rding inequality (see Theorem~\ref{thm_Gardingcomp_rhodelta}):

\begin{theorem}
\label{thm_Gardingcomp}
Let $G$ be a connected compact Lie group. 
Let $m\in \bR$.
	Assume that the symbol $\sigma\in S^{m}(G)$ satisfies the elliptic condition $\sigma \geq c_0 (\id  +\widehat \cL)^{\frac m 2}$ for some constant $c_0>0$.
	Then there exist  constants $c,C>0$ such that 
	$$
	\forall f\in \cC^\infty(G),\qquad 
	\Re \left (\Op^{\rm KN}(\sigma) f, f\right )_{L^2(G)} \geq c\| f\|_{H^{\frac m2}(G)}^2-C \|f\|_{H^{\frac{m-1}2}(G)}^2 .
$$	 
\end{theorem}
Above, the spaces $H^m(G)$ denote the usual Sobolev spaces defined on any compact manifold, here $G$, 
while the definitions of  the symbol classes $S^m(G)$ and  the Laplace-Beltrami operator $\cL$ as well as its Fourier transform are recalled in Section \ref{subsec_pseudoC_compact}.

\smallskip

The next result concerns  the symbolic pseudodifferential calculus on a graded Lie group $G$ \cite{FR,FF0}, briefly recalled in Section \ref{subsec_pseudoC_graded}.
The Sobolev space $L^2_s(G)$ will be the ones adapted to this setting \cite{FR,FRSob}. We will prove the $(\rho,\delta)$-generalisation of the following G\aa rding inequality (see Theorem~\ref{thm_Gardingnilp_rhodelta}):
\begin{theorem}
\label{thm_Garding_graded}
Let $G$ be a graded nilpotent Lie group. 
Let $m\in \bR$.
	Assume that  $\sigma\in S^{m}(G)$ satisfies the elliptic condition $\sigma \geq c_0 (\id  +\widehat \cR)^{\frac {m}\nu}$
	for some constant $c_0>0$ where $\cR$ is a positive Rockland operator of homogeneous degree $\nu$.
	Then there exist  constants $c,C>0$ such that 
	$$
	\forall f\in \cC^\infty_c(G),\qquad 
	\Re \left (\Op^{\rm KN}(\sigma) f, f\right )_{L^2(G)} \geq c\|f\|_{L^2_{\frac{m}2}(G)}^2 -C \|f\|_{L^2_{\frac{m-1}2}(G)}^2 .
$$	 
\end{theorem}

Still in the context of graded Lie groups, 
our method is particularly adapted to  the semi-classical counter-part of Theorem \ref{thm_Garding_graded}:
 \begin{theorem}\label{thm:garding_sc}
 Let $\sigma\in \cA_0$, that is, the symbol $\sigma$ is a smoothing symbol with $x$-compact support. 
If $\sigma$ is non-negative, then  there exists a constant $C>0$ such that 
 \begin{equation}\label{eq:minnie}
 \forall f\in L^2(G), \quad \forall \eps\in (0,1],\qquad 
 \Re \left(\Op_\eps(\sigma) f, f\right) _{L^2(G)}\geq  -C\eps \| f\|^2_{L^2(G)}.
 \end{equation}
 \end{theorem}
 
 This inequality is exactly what is used in the Euclidean setting for proving the positivity of semi-classical measures (see~\cite{gerard_leichtnam}). However, in $\R^n$, one can prove a stronger result, known as {\it sharp G\aa rding inequality}, in which the right-hand side in~\eqref{eq:minnie}  is $-C\eps \| f\|^2_{H^{-1/2}}$, under the assumptions of Theorem~\ref{thm:garding_sc} (see~\cite{Zwobook}). The proof of such an estimate requires to use other tools than the sole Wick quantization.
The semi-classical calculus in this setting \cite{FF1,FF2}
is recalled in Section \ref{subsec_semiclC}, 
as well as the definitions for $\cA_0$ and $\Op_\eps$.

\medskip

On a compact Lie group $G$, the pseudodifferential calculus  mentioned above with $\rho>\delta$ and $\rho\geq 1-\delta$ coincides with H\"ormander's pseudodifferential calculus defined via charts on the compact Lie group $G$ viewed as a compact manifold \cite{FJFA2015,RTW}.
However, the notion of symbols are not the same in these two calculi: the one presented here or in \cite{FJFA2015,RT,RTW}
is  global and  based on the representations of the group.
For a graded nilpotent Lie group $G$, 
the pseudodifferential calculus mentioned above coincides 
with the (global) H\"ormander calculus only when 
$G$ is abelian, that is, only when $G$ is the abelian group $(\bR^n,+)$ with $n=\dim G$.
Otherwise, although a graded nilpotent Lie group is globally diffeomorphic to $\bR^n$ as a manifold, 
the calculi will not be comparable.
At this point, we ought to  clarify what we mean by pseudodifferential calculus on a  smooth manifold $M$ 
in this paper:

\begin{definition}
\label{def_pseudo-diff_calculus}
For each $m\in \bR$, 
let $\Psi^m(M)$  be a given Fr\'echet space of continuous operators $\cD(M)\to\cD(M)$.
We say that the space $\Psi^\infty(M):=\cup_m \Psi^m(M)$ form a \emph{pseudodifferential calculus} 
when it is an algebra of operators satisfying: 
\begin{enumerate}
\item 
\label{item_def_pseudo-diff_calculus_inclusion}
The continuous inclusions
$\Psi^m (M)\subset \Psi^{m'}(M)$ hold for any $m\leq m'$.
\item
\label{item_def_pseudo-diff_calculus_product}
$\Psi^\infty(M)$ is an algebra of operators.
Furthermore if $T_1\in \Psi^{m_1}(M)$, $T_2\in \Psi^{m_2}(M)$, 
then $T_1T_2\in \Psi^{m_1+m_2}(M)$,
and the composition is continuous as a map
$\Psi^{m_1}(M)\times \Psi^{m_2}(M)\to \Psi^{m_1+m_2}(M)$.
\item
\label{item_def_pseudo-diff_calculus_adjoint}
$\Psi^\infty(M)$ is stable under taking the adjoint.
Furthermore if $T\in \Psi^{m}(M)$  then $T^*\in \Psi^{m}(M)$,
and taking the adjoint is continuous as a map
$\Psi^{m}(M)\to \Psi^{m}(M)$.
 \end{enumerate}
 \end{definition}

The paper is organised as follows. We start with recalling the definition of the Kohn-Nirenberg quantization on groups and introducing Wick's   (Section \ref{sec_quantization}).
Then we show that a link between these two quantizations in the symbolic calculi provides a proof of 
G\aa rding inequalities in the cases of compact Lie groups $G$ 
(Section \ref{sec_Gcomp}), and of graded nilpotent Lie groups $G$ (Section \ref{sec_Gnilp}).
In the latter case, we also study the semi-classical analogue in Section \ref{sec_SC_nilpG} and discuss in that setting the consequences of the Wick-calculus in terms of frame. 
In the Appendix, we develop the same strategy of proof in the Euclidean case; to our knowledge, the  proof which is the closest to ours is Folland's one~\cite[Chapter 2, Section 6]{folland}.

\medskip

\noindent{\bf Acknowledgements}. The authors warmly thank 
Serena Federico for interesting and useful discussions on G\r arding inequalities.
  CFK thanks the Erwin Schr\"odinger Institut for its hospitality when writing this paper. LB and CFK benefit from the support of  the Région Pays de la Loire via the Connect Talent Project HiFrAn 2022 07750, and  from the France 2030 program, Centre Henri Lebesgue ANR-11-LABX-0020-01. CFK and VF acknowledge the support of  The Leverhulme Trust via Research Project Grant RPG 2020-037. 
  \medskip 

  \noindent{\bf Notation}. We use the notation $f\lesssim g$ when there exists a constant $C>0$ such that $f\leq Cg$. Moreover, when $f\lesssim g$ and $g\lesssim f$, we will write $f\sim g$. We will write $f\lesssim_G g$ when the constant involved in the estimate depends on $G$.

\section{Quantizations on groups}
\label{sec_quantization}

In this section, we discuss two quantizations procedures on groups that are based on the group Fourier transform  and the associated Plancherel theorem. These latter notions require some hypotheses on the group we now list. The group  $G$ is a separable  locally compact group.
We assume that it is unimodular, that is, its left (resp. right) Haar measures are also right (resp. left) invariant. 
We also assume that it is of type I.
The paper may be read without understanding these technical hypotheses. It suffices to know that  they ensure that the Plancherel theorem  holds, and that 
 they  are naturally satisfied on Lie groups that are compact or nilpotent. 
 
\subsection{Fourier analysis}

\subsubsection{The dual set}
Recall that a (unitary) \emph{representation} $(\mathcal H_\pi, \, \pi)$ of $G$ is a pair consisting in a Hilbert space~$\mathcal H_\pi$  and a group morphism~$\pi$ from~$G$ to the set of unitary operators on $\mathcal H_\pi$.
In this paper, the representations will always be assumed (unitary) strongly continuous, and their associated Hilbert spaces separable. 
A representation is said to be {\it irreducible} if the only closed subspaces of $\mathcal H_\pi$ that are stable under~$\pi$ are $\{0\}$ and $\mathcal H_\pi$ itself. 
Two representations $\pi_1$ and $\pi_2$ are equivalent  if there exists a unitary transform $\mathbb U$ called an {\it intertwining map} that sends $\mathcal H_{\pi_1}$ on $\mathcal H_{\pi_2}$ with 
$$\pi_1=\mathbb U^{-1}\circ  \pi_2 \circ \mathbb U.$$ 
The {\it dual set} $\widehat G$ is obtained by taking the quotient of the set of irreducible representations by this equivalence relation.  We may still denote by $\pi$ the elements of $\widehat G$ and we keep in mind that different representations of the class are equivalent through intertwining operators. 

\subsubsection{Fixing a Haar measure}
We fix a Haar measure that we denote by $dx$ when the variable of integration is $x\in G$ or by $dy$ if the variable is $y$.  

The non-commutative convolution is given via
\begin{equation}\label{def:convolution}
 (f_1*f_2)(x):=\int_G f_1(y) f_2(y^{-1}x) dy,\quad x\in G
\end{equation}
for $f_1,f_2\in {\mathcal C}_c(G)$; here 
${\mathcal C}_c (G)$ denotes the space of continuous complex-valued functions on $G$ with compact support.

\subsubsection{The Fourier transform}
\label{subsubsecFG}
The {\it Fourier transform} of an integrable function $f\in L^1(G)$ at a representation $\pi$ of $G$ is the operator acting on $\mathcal H_\pi$
 via 
 $$
 \widehat f(\pi):=
 \cF(f)(\pi) :=\int_G  f(z)\, (\pi(z))^*\, dz.
 $$
 Note that if $f_1,f_2\in \cC_c (G)$ then 
 \begin{equation}
 \label{eq_cFf1*f2}
 \widehat {f_1 * f_2} = \widehat f_2 \widehat f_1.
 \end{equation}

 If $\pi_1,\pi_2$ are two equivalent representations of $G$ with $\pi_1=\mathbb U^{-1}\circ  \pi_2 \circ \mathbb U$ for some intertwining operator $\mathbb U$, then 
$$\mathcal F(f)(\pi_1) =\mathbb U^{-1}\circ \mathcal F(f)(\pi_2)\circ \mathbb U .
$$
Hence, this defines the measurable field of operators $\{{\mathcal F}(f)(\pi), \pi\in \Ghat \}$ modulo equivalence. 
The unitary dual $\Ghat$ is equipped with its natural Borel structure, 
and the equivalence comes from quotienting the set of irreducible representations of $G$ together with understanding the resulting fields of operators modulo intertwiners.

\subsubsection{The Plancherel Theorem} 

Here, we recall the Plancherel Theorem due to Dixmier \cite[Ch. 18]{dixmier}. Among other results, 
it states the existence and uniqueness of the {\it Plancherel measure}, that is, the positive Borel measure $\mu$ 
on $\widehat G$ such that 
 the Plancherel formula
\begin{equation}
  \label{eq_plancherel_formula}
\|f\|^2_{L^2(G)}=\int_G |f(x)|^2 dx = \int_{\widehat G} \|\widehat f(\pi)\|_{HS(\mathcal H_\pi)}^2 \,d\mu(\pi),
\end{equation}
 or equivalently
 $$
 (f_1,f_2)_{L^2(G)}=\int_G f_1(x)\overline{f_2 (x)}dxdx = \int_{\widehat G} {\rm tr}_{\mathcal H_\pi} \left(\widehat f_1(\pi)\widehat f_2(\pi)^*)\right) \,d\mu(\pi)
 $$
holds for any $f\in {\mathcal C}_c(G) $.
Here $\|\cdot\|_{HS(\mathcal H_\pi)}$ denotes the Hilbert-Schmidt norm on $\mathcal H_\pi$.  
This implies that the group Fourier transform 
is a unitary map from $L^1(G)\cap L^2(G)$ equipped with the norm of $L^2(G)$ to the Hilbert space
$$
L^2(\widehat G):=\int_{\widehat G} \mathcal H_\pi \otimes\mathcal H_\pi^*\, d\mu(\pi).
$$
We identify $L^2(\Ghat)$ with the space of $\mu$-square integrable Hilbert-Schmidt fields on $\widehat G$;  its Hilbert norm and scalar products are then given by 
\begin{align*}
	\|\tau\|_{L^2(\Ghat)}^2 &= \int_{\Ghat} \|\tau (\pi)\|_{HS(\mathcal H_\pi)}^2 \,d\mu(\pi), \quad \tau\in L^2(\Ghat), \\
(\tau_1,\tau_2)_{L^2(\Ghat)} 
&= \int_{\Ghat} \tr_{\cH_\pi} (\tau_1 (\pi)\ \tau_2(\pi)^*) \,d\mu(\pi), \quad \tau_1,\tau_2\in L^2(\Ghat).
\end{align*}
Here $\tr_{\cH_\pi}$ denotes the trace of operators on the Hilbert space $\cH_\pi$.
The group Fourier transform $\cF$ 
extends unitarily from 
 $L^2(G)$ onto $L^2(\widehat G)$.

We denote by $L^\infty(\Ghat)$ the space of measurable fields (modulo equivalence) of bounded operators
$\sigma = \{\sigma(\pi)\in \sL(\cH_\pi) 
: \pi\in \Ghat\}$ on $\Ghat$ such that 
$$
\|\sigma\|_{L^\infty(\Ghat)} := \sup_{\pi\in \Ghat} \|\sigma(\pi)\|_{\sL(\cH_\pi) }
$$
is finite; here the supremum refers to the essential supremum with respect to the Plancherel measure~$\mu$ of $\Ghat$. 
In fact, $L^\infty(\Ghat)$ is naturally a Banach space and moreover a von Neumann algebra, sometimes called the von Neumann algebra of the group $G$. 
It acts naturally on $L^2(\Ghat)$ by composition on the left: 
$$
(\sigma \tau) (\pi) = \sigma(\pi) \ \tau(\pi),\quad \pi\in \Ghat, \ \sigma\in L^\infty(\Ghat) \ \mbox{and}\ \tau \in L^2(\Ghat),
$$
(it also acts on the right)
and this action is continuous
$$
\|\sigma\tau\|_{L^2(\Ghat)}\leq 
\|\sigma\|_{L^\infty(\Ghat)}
\|\tau\|_{L^2(\Ghat)}.
$$
Dixmier's Plancherel theorem implies that 
$L^\infty(\Ghat) $ is isomorphic to the von Neumann algebra $\sL(L^2(G))^G$ of linear bounded operators on $G$ that are invariant under left translations. 
The isomorphism is given by the fact that 
the Fourier multiplier with symbol $\sigma$, i.e. the operator $f\mapsto \cF^{-1} (\sigma \widehat f)$, is an operator  in  $\sL(L^2(G))^G$.
\smallskip 

Note that 
 $\cF L^1(G)\subseteq L^\infty(\Ghat)$ with 
$$
\forall f\in L^1(G),\qquad 
\|\widehat f\|_{L^\infty(\Ghat)} \leq \|f\|_{L^1(G)}.
$$

\subsection{The  Kohn-Nirenberg quantization}

In this section, we recall some results related to  the symbolic quantization on groups introduced by Michael Taylor \cite{Taylor}.
When $G$ is the abelian group~$\bR^n$, this is the quantization often used in the field of Partial Differential Equations and called the Kohn-Nirenberg quantization or classical quantization~\cite{ho,AG}.
 We keep this vocabulary in the group case. 

\subsubsection{The space $L^2(G\times \Ghat)$}
We may identify the tensor product $L^2(G)\otimes L^2(\Ghat)$ with the space denoted by $L^2(G\times \Ghat)$ of measurable fields $\tau=\{\tau(x,\pi) \in HS(\cH_\pi) \ : \ (x,\pi)\in G\times \Ghat\}$ of Hilbert-Schmidt operators (up to equivalence) such that the quantities 
$$
\|\tau\|_{L^2(G\times \Ghat)}^2
:=
\int_{G\times\Ghat} \|\tau(x,\pi)\|_{HS(\cH_\pi)}^2 dxd\mu(\pi)
$$
are finite. It is naturally a separable Hilbert space with norm $\|\cdot\|_{L^2(G\times \Ghat)}$ and scalar product given by
$$
(\tau_1,\tau_2)_{L^2(G\times \Ghat)}
=
\int_{G\times\Ghat} \tr_{\cH_\pi} \left( \tau_1(x,\pi) \ \tau_2(x,\pi)^*\right)  dxd\mu(\pi),
\quad \tau_1,\tau_2\in L^2(G\times \Ghat).
$$

By the Plancherel theorem, the Hilbert space $L^2(G\times \Ghat)$ and $L^2(G\times G)$ are isomorphic via the Fourier transform in the second variable:
$$
L^2(G\times G)\longrightarrow
L^2(G\times \Ghat) ,\qquad 
\kappa 
 \longmapsto (\id\otimes \cF) \kappa.
$$
In other words, 
any  $\tau\in L^2(G\times \Ghat)$ may be written as 
$$
\tau (x,\pi) =\widehat \kappa_{\tau,x}(\pi) 
$$ for a unique function $\kappa_\tau:(x,y)\mapsto \kappa_{\tau,x}(y) = \kappa_{\tau}(x,y)$   in $L^2(G\times G)$. 

\subsubsection{The quantization $\Op^{\rm KN}$ on $L^2(G\times \Ghat)$}
\label{subsubsec_OPKN}
For any $f\in \cC_c(G)$ and $\tau\in L^2(G\times \Ghat)$, the symbol
$$
\tau \widehat f := \{\tau(x,\pi) \pi(f) \ : \ (x,\pi)\in G\times \Ghat\}
$$
is measurable on $G\times \Ghat$ and  
satisfies
$$
\|\tau \widehat f\|_{L^2(G\times \Ghat)} \leq \|\tau \|_{L^2(G\times \Ghat)} \|\widehat f\|_{L^\infty(\Ghat)}.
$$
Hence $\tau \widehat f\in L^2(G\times \Ghat)$ and 
 we can define
$(\id\otimes \cF^{-1}) (\tau \widehat f) \in L^2(G\times G)$. 
By \eqref{eq_cFf1*f2}, 
we have:
$$
(\id\otimes \cF^{-1}) (\tau \widehat f)(x,z)
= f*\kappa_{\tau,x}(z)=
 \int_G f(y) \kappa_{\tau,x}(y^{-1}z)dy,\qquad (x,z)\in G\times G
.
$$
As $f\in \cC_c(G)$, $(\id\otimes \cF^{-1}) (\tau \widehat f)(x,z)$ is in fact continuous in $z$ and it makes sense to define:
\begin{equation}
\label{eq_OpKN}
\Op^{\rm KN}(\tau) f (x):= (\id\otimes \cF^{-1}) (\tau \widehat f) (x,x)=
 f*\kappa_{\tau,x}(x),\qquad f\in \cC_c(G), \, x\in G. 
\end{equation}

It follows from the formula above that the integral kernel of $\Op^{\rm KN}(\tau)$ is given by 
$$
G\times G\ni (x,y)\longmapsto \kappa_{\tau,x}(y^{-1}x).
$$
Hence,  the operator $\Op^{\rm KN}(\tau)$ extends uniquely into a Hilbert-Schmidt operator on $L^2(G)$ with norm
$$
\|\Op^{\rm KN}(\tau)\|_{HS(L^2(G))}
=
\|\kappa_\tau\|_{L^2(G\times G)}
=
\|\tau\|_{L^2(G\times\Ghat)}
$$
Consequently, $\Op^{\rm KN}$ is an isometry from $L^2(G\times \Ghat)$ onto $HS(L^2(G))$.

\subsubsection{Extension of $\Op^{\rm KN}$ to $\cC(G, \cF L^1(G))$}

We can extend naturally $\Op^{\rm KN}$ via \eqref{eq_OpKN} to the space $\cC(G, \cF L^1(G))$, that is, to the symbols 
$\sigma$ of the form $\sigma(x,\pi)=\widehat \kappa_x(\pi)$ with convolution kernel $\kappa\in \cC(G,L^1(G))$.
By injectivity of the Fourier transform, 
the two possible definitions of $\Op^{\rm KN}$ on symbols in $L^2(G\times \Ghat)$ and  $\cC(G, \cF L^1(G))$ coincide. 
Note that $\Op^{\rm KN}(\sigma)$ with $\sigma\in \cC(G, \cF L^1(G))$ will act on $\cC_c(G)$ and the Young convolution inequality implies  the following estimate for the operator norm as operators on  $L^2(G)$. 
\begin{lemma}
\label{lem_A0norm}
	If $\sigma\in \cC(G,\cF L^1(G))$ then 
	$$
	\|\Op^{\rm KN} (\sigma)\|_{\sL(L^2(G))} \leq 
	\int_G \sup_{x\in G} |\kappa_x(y)| dy 
	$$
\end{lemma}
\begin{proof}
Let $\kappa\in \cC(G,L^1(G))$ and  $f\in \cC_c(G)$. 
We have
$$
	|f*\kappa_x(x)|\leq |f|*\sup_{x'\in G}|\kappa_{x'}| (x),
	$$	
so the Young convolution inequality yields
	$$
	\sqrt{\int_G |f*\kappa_x(x)|^2 dx} \leq \left\| |f|*\sup_{x'\in G}|\kappa_{x'}|\right\|_{L^2(G)}
	\leq \|f\|_{L^2(G)} \left\|\sup_{x'\in G}|\kappa_{x'}|\right\|_{L^1(G)}.
	$$	
\end{proof}

If $\sigma\in \cC(G,\cF L^1(G))$, we define
\begin{equation}
	\label{eq_A0norm}
\|\sigma\|_{\cA_0}
:=\int_G \sup_{x\in G} |\kappa_x(y)| dy, 
\qquad \sigma(x,\pi)=\cF\kappa_x(\pi).
\end{equation}
We denote by $\cC_b(G,\cF L^1(G))$ 
the subspace of those
$\sigma\in\cC(G,\cF L^1(G))$  such that $\|\sigma\|_{\cA_0}
$ is finite. 
We also denote by $\cC_b(G,L^1(G))$ the space of $\kappa\in \cC(G,L^1(G))$ such that $\int_G \sup_{x\in G} |\kappa_x(y)| dy$ is finite.

\subsubsection{Extension of $\Op^{\rm KN}$ via the inversion formula}

We can  extend $\Op^{\rm KN}$ to a larger space of symbols than $\cC(G,\cF L^1(G))$ but acting on a smaller space of functions than $\cC_c(G)$ under some further technical assumptions. 
Indeed, let us consider
the space $\cC(G,L^\infty(\Ghat))$ of symbols  $\sigma$ that are continuous maps from $G$ to $L^\infty(\Ghat)$. 
A symbol $\sigma$ in $\cC(G,L^\infty(\Ghat))$  is naturally identified  with a measurable field (up to equivalence) of operators $\sigma=\{\sigma(x,\pi) \in \sL(\cH_\pi) \, : \, (x,\pi)\in G\times \Ghat\}$ satisfying conditions of continuity in $x$ and boundedness in $\pi$. 
We also consider, when it exists, a  space $S$ of bounded, continuous and integrable functions satisfying:
\begin{itemize}
	\item[(i)] the space $S\cap \cC_c(G)$ is dense in $L^2(G)$,  
	\item[(ii)] for any $f\in S$, the operators $\widehat f(\pi)$, $\pi\in \Ghat$,
are trace-class and 
the following quantity is finite:
$$
\int_{\Ghat} \tr_{\cH_\pi} |\widehat f(\pi)|d\mu(\pi)<\infty.
$$
\end{itemize} 	
As a consequence of the Plancherel formula, 
the following inversion formula holds:
$$
 f(x)
=  \int_{\widehat G} \tr_{\cH_\pi} \, \Big(\pi(x) \widehat  f(\pi)  \Big)\, d\mu(\pi) , \quad f\in S, \ x\in G,
$$
provided that $G$ is amenable. 
We will not discuss here 
these technical assumptions (existence of $S$ and amenability of $G$), but just comment on the fact that they are naturally satisfied for compact or nilpotent Lie groups with $S$ being the space of smooth functions with compact support; in the nilpotent case, we can take $S$ to be the space of Schwartz functions. 
With the inversion formula, $\Op^{\rm KN}$ extends to the quantization given for symbols $\sigma$ in $\cC(G,L^\infty(\Ghat))$ by:
$$
\Op^{\rm KN}(\sigma)f(x) = \int_{\Ghat} \tr_{\cH_\pi} \, \Big(\pi(x)\sigma(x,\pi) \widehat f(\pi)  \Big)\, d\mu(\pi) , \quad f\in S, \ x\in G.
$$
Naturally, this coincides with the quantization defined above for $\sigma\in L^2(G\times \Ghat)$ and for $\sigma \in \cC(G,\cF L^1(G))$.

At least formally, the integral kernel of 
$\Op^{\rm KN}$ is made explicit  when writing 
\begin{equation}\label{eq:calamity}
\Op^{\rm KN}(\sigma) f (x)= \int_{G\times \widehat G} \tr_{\cH_\pi} \left( \sigma(x,\pi) \pi(y^{-1}x) \right) f(y) dy d\mu(\pi).
\end{equation}

\subsection{The Wick quantization}\label{sec:wick}

Another natural symbolic quantization appears on the (locally compact, unimodular, type I) group $G$, in the same flavour as Wick's quantization (see~\cite{Lerner}).
For this, we start by defining the transformation $\cB=\cB_a$ associated with a continuous, square-integrable and bounded function $a$ satisfying $\| a\|_{L^2(G)}=1$.

On $\R^n$, the natural choice for  such a function $a$ is a Gaussian,
or a family of Gaussian re-scaled with a small parameter (see~\cite{corobook}). In the examples we treat in the next sections, 
a similar choice  will be to  consider a family of  functions $a=a_t$ which are the heat kernels at time $t$ when the group is equipped with a Laplace-like operator, as in Section~\ref{sec_Gcomp}, or the $t$-rescaling of a given function when the group is equipped with dilations (see Sections~\ref{sec_Gnilp} and~\ref{sec_SC_nilpG} below).

\subsubsection{The transformation $\cB$}\label{sec:B}
\label{subsubsec_B}
First, for each $(x,\pi)\in G\times \Ghat$, we define the  operator on $\cH_\pi$ depending on $y\in G$,
$$
\sF_{x,\pi}(y)=  a(x^{-1}y)\pi(y)^*.
$$
We check readily that $\sF_{x,\cdot} \in \cC(G,L^\infty(\Ghat ))$ with
$$
\sup_{y\in G} \| \sF_{x,\cdot} (y)\|_{L^\infty(\Ghat)} \leq  \| a\|_{L^\infty(G)}.
$$
We can now define the operator $\cB=\cB_a$ on $\cC_c(G)$ via
$$
\cB[f](x,\pi)=  \int_G f(y) \, \sF_{x,\pi}(y) dy,\quad f\in \cC_c(G), \ (x,\pi)\in G\times \Ghat.
$$
We observe that $\cB[f]$ is the field of operators  on $G\times \Ghat$ given by 
\begin{equation}\label{eq:B}
\cB[f] (x,\pi)= \cF\left( f \, a(x^{-1}\,\cdot)\right)(\pi), \qquad  (x,\pi)\in G\times \Ghat. 
\end{equation}

\begin{remark}\label{rem:luke}
    In the case of $G=\bR^n$, we have $\Ghat=\bR^n$ and $\mathcal H_\pi=\C$. Therefore, for all $a\in L^2(\bR^n)$, the function $ \mathcal F_{x,\pi} (y)$ is scalar-valued. It coincides with the wave packets  defined in~\cite{corobook}. Moreover, if  $a$ is chosen as a Gaussian function, we recognize  $\cB$ as the Bargmann transform~\cite{folland,Lerner,corobook}.
    This explains the notation. 
\end{remark}

The map $\mathcal B$ has  frame's properties:

\begin{proposition}
\label{prop_B}
\begin{enumerate}
	\item For any $f\in \cC_c(G)$, $\cB[f]$ defines an element of $L^2(G\times \Ghat)$ with norm 
	$$
	\|\cB[f]\|_{L^2(G\times \Ghat)} = \|f\|_{L^2(G)}
	$$
 and the map $\cB$ extends uniquely to an isometry from $L^2(G)$ to
$L^2(G\times \widehat G)$ for which we keep the same notation. 
	\item The adjoint map
$\cB^* : L^2(G\times \widehat G)\to L^2(G)$ is given by 
$$
\cB^*[\tau](y)= \int_{G\times \Ghat}  \tr_{\cH_\pi} \left( \tau(x,\pi) (\sF_{x,\pi}(y))^* \right) dxd\mu(\pi), 
\qquad \tau\in L^2(G\times \Ghat), \ y\in G,
$$
in the sense that for any $f\in L^2(G)$, 
$$
(\cB^*[\tau],f)_{L^2(G)}
=\int_{G\times \Ghat}  \tr_{\cH_\pi} \left( \tau(x,\pi) \left(\cF\left ( fa(x^{-1}\,\cdot)\right )(\pi)\right )^* \right) dxd\mu(\pi).
$$
If $\tau = (\id\otimes\cF) \kappa$, $\kappa\in L^2(G\times G)$, then 
\begin{equation}\label{eq:2.75}
\cB^* [\tau](y) = \int_G \kappa_x(y) \bar  a(x^{-1}y) dx
=  (\kappa_{\, \cdot} (y)*\bar a) (y).    
\end{equation}
\item We have $\cB^*\cB=\id_{L^2(G)}$ while $\cB \cB^*$ is a projection on a closed subspace of $L^2(G\times \Ghat)$.
\end{enumerate}
\end{proposition}

\begin{proof}
From~\eqref{eq:B} and the Plancherel formula \eqref{eq_plancherel_formula}, we obtain
$$
\int_{\Ghat}
\|\cB[f](x,\pi)\|_{HS(\cH_\pi)}^2
d\mu(\pi) = \|f\, a(x^{-1}\,\cdot)\|_{L^2(G)}^2, 
\qquad x\in G.
$$
Integrating against $dx$ yields Part (1). 

\smallskip 

Part (2) follows from 
$$
\int_{G\times \Ghat}  \tr_{\cH_\pi} \left( \tau(x,\pi) \left(\cF\left ( fa(x^{-1}\,\cdot)\right )(\pi)\right )^* \right) dxd\mu(\pi)
=
\int_{G\times \Ghat}  \tr_{\cH_\pi} \left( \tau(x,\pi) \left(\cB[f](x,\pi)\right )^* \right) dxd\mu(\pi),
$$
by \eqref{eq:B}.

Part 3 follows from Part (1) since it implies for any $f,g\in L^2(G)$
$$
(f,g)_{L^2(G)}= (\cB [f],\cB[g])_{L^2(G\times \Ghat)}=(\cB^*\cB[f],g)_{L^2(G)}.
$$
\end{proof}

As a corollary, considering in each space $\mathcal H_\pi$ an orthonormal basis
$(\varphi_k(\pi))_{k\in I_\pi}$, where $I_\pi\subset \N$,
we obtain an integral representation of square integrable functions as a superposition of  wave packets (see~\cite{FF1,FL}). Set for $(x,\pi)\in G\times \widehat G$ and $k,\ell\in I_\pi$,
\[
g_{x,\pi,k,\ell}(y):=  \left(\sF_{x,\pi}(y)^*\varphi_k(\pi),\varphi_\ell(\pi)\right)_{\cH_\pi},\;y\in G,
\]
where $(\cdot,\cdot)_{\cH_\pi}$ denotes the inner product of $\mathcal H_\pi$.
The frame properties in Proposition \ref{prop_B} (3) implies the following  decomposition.

\begin{corollary}\label{cor:frame}
A function  $f\in L^2(G)$ decomposes  in $L^2(G)$ as
    \[
f=\int_{G\times \widehat G}
\sum_{k,\ell\in I_\pi} \left( f, g_{x,\pi,k,\ell}\right)_{L^2(G)} \; g_{x,\pi,k,\ell} \; dxd\mu(\pi),
    \]
    in the sense that 
    $$
    \|f\|_{L^2(G)}^2 = \int_{G\times \Ghat} 
    \sum_{k,\ell\in I_\pi} 
    |\left( f, g_{x,\pi,k,\ell}\right)_{L^2(G)} |^2 dx d\mu(\pi),
    $$
    or equivalently for any $f_1,f_2\in L^2(G)$
    $$
    (f_1,f_2)_{L^2(G)}
    =\int_{G\times \widehat G}
\sum_{k,\ell\in I_\pi} \left( f_1, g_{x,\pi,k,\ell}\right)_{L^2(G)} \;
\overline{\left( f_2, g_{x,\pi,k,\ell}\right)_{L^2(G)}}
 \; dxd\mu(\pi).
    $$
\end{corollary}

\begin{proof}
By Proposition \ref{prop_B}, we have for any $f\in L^2(G)$
\begin{align*}
   \|f\|^2_{L^2(G)}
=  \|\cB[f]\|^2_{L^2(G\times \Ghat)}= \int_{G\times \Ghat} \|\cB[f](x,\pi) \|_{HS(\cH_\pi)}^2 dx d\mu(\pi). 
\end{align*}
The Hilbert-Schmidt norms may be written in the basis $(\varphi_k)$ as
$$
\|\cB[f](x,\pi) \|_{HS(\cH_\pi)}^2
=
\sum_{k,\ell} |(\cB[f](x,\pi) \varphi_k,\varphi_\ell)_{\cH_\pi}|^2
$$
 with
\begin{align*}
(\cB[f](x,\pi) \varphi_k,\varphi_\ell)_{\cH_\pi} 
&= \int_G f(y) (\sF_{x,\pi}(y) \varphi_k,\varphi_\ell)_{\cH_\pi} dy\\
&= \int_G f(y)\overline{ (\varphi_\ell, \sF_{x,\pi}(y) \varphi_k)}_{\cH_\pi} dy\\
&= \int_G f(y)\overline{ (\sF_{x,\pi}(y)^* \varphi_\ell, \varphi_k)}_{\cH_\pi} dy= (f,g_{x,\pi,\ell,k})_{L^2(G)}.
\end{align*}
We then conclude on $  (f_1,f_2)_{L^2(G)}$ by considering $\| f_1\pm f_2\|^2$ and $\| f_1\pm if_2\|^2$. 
\end{proof}

\subsubsection{The quantization $\Op^{\rm Wick}$}
We can now define the Wick quantization $\Op^{\rm Wick} = \Op^{{\rm Wick},a}$. 
It depends on the function $a$ fixed at the beginning of the  Section~\ref{sec:wick} with $\|a\|_{L^2(G)}=1$.
We set
$$
\Op^{\rm Wick} (\sigma) f= \cB^* \sigma \cB[f], 
\qquad f\in L^2(G), \ \sigma\in L^\infty(G\times \Ghat).
$$
Here, $L^\infty(G\times \Ghat)$ denotes the space 
of symbols $\sigma=\{\sigma(x,\pi) \ : \ (x,\pi)\in G\times \Ghat\}$ which are bounded in $(x,\pi)\in G\times \Ghat$, 
i.e. a measurable field of operators in $(x,\pi)\in G\times \Ghat$ such that 
$$
\|\sigma\|_{L^\infty(G\times\Ghat)} : =\sup_{(x,\pi)\in G\times \Ghat} \|\sigma(x,\pi)\|_{\sL(\cH_\pi)}
$$
is finite, the supremum referring to the essential supremum for the measure $dxd\mu$ on $G\times \Ghat$. 
This is naturally a Banach space (even a von Neumann algebra). 
Moreover, it acts naturally continuously on $L^2(G\times \Ghat)$ by left  composition (and also 
right composition) with
$$
\|\sigma \tau \|_{L^2(G\times \Ghat)} 
\leq
\|\sigma  \|_{L^\infty(G\times \Ghat)}
\| \tau \|_{L^2(G\times \Ghat)}, 
\qquad \sigma\in L^\infty(G\times \Ghat), \ 
\tau\in L^2(G\times \Ghat).
$$
This implies that the quantization $\Op^{\rm Wick}$ is well defined:
 
\begin{proposition}
	The symbolic quantization $\Op^{\rm Wick}$ is well defined on $L^\infty(G\times \Ghat)$ and satisfies
$$
\forall \sigma\in L^\infty(G\times\Ghat),\qquad 
\|\Op^{\rm Wick}(\sigma)\|_{\sL(L^2(G))}\leq \|\sigma\|_{L^\infty(G\times \Ghat)}.
$$	
\end{proposition}

\begin{proof}
We have for any $f\in L^2(G)$:
\begin{align*}
	\|\Op^{\rm Wick} (\sigma)f\|_{L^2(G)}
&=\|\cB^* \sigma \cB[f]\|_{L^2(G)}
\\&\leq \|\cB^*\|_{\sL(L^2(G\times \Ghat),L^2(G))}
\|\sigma\|_{L^\infty(G\times \Ghat)}
\|\cB\|_{\sL(L^2(G), L^2(G\times \Ghat))}\|f\|_{L^2(G)}.
\end{align*}
Since $\cB$ is an isometry, the operator norms of $\cB$ and $\cB^*$ are equal to 1. 
\end{proof}

\begin{remark}
	In the case of $G=\bR^n$ as in Remark~\ref{rem:luke}, and for $a$ being  a Gaussian function, we recognize  $\Op^{\rm Wick}$ as the Wick quantization \cite{folland,Lerner,corobook}. 
 \end{remark}

As an example, we observe that  Proposition~\ref{prop_B} (3) may be rephrased as  $\Op^{\rm Wick}(\id) =\id_{L^2(G)}$ where $\id$ is the symbol $\id = \{ \id_{\cH_\pi} , (x,\pi)\in G\times \Ghat)\}$. 

\medskip

The following computation will allow for the comparison  between the Wick and Kohn-Nirenberg quantizations on  $\cC_b(G,\cF L^1(G))$; note that a symbol in $\cC_b(G,\cF L^1(G))$ is in $L^\infty (G\times 
\Ghat)$.
\begin{lemma}
\label{lem_kW}
	If a symbol $\sigma$ is in $\cC_b(G,\cF L^1(G))$, 
	then 
	$$
	\Op^{\rm Wick} (\sigma) f (x) = f * \kappa^{\rm Wick}_x(x), \qquad f\in \cC_c(G), \, x\in G,
	$$
	where $\kappa^{\rm Wick}\in \cC_b(G, L^1(G))$ is given by:
	\begin{align*}
\kappa^{\rm Wick}_x(w) 
&=
\int_G a(z^{-1}xw^{-1})\bar a(z^{-1}x) \kappa_z(w) dz	
\\
&=
\int_G a(z' w^{-1})\bar a(z') \kappa_{xz'^{-1}}(w) dz',
\end{align*}
	and $\kappa\in \cC_b(G,L^1(G))$ is the convolution kernel of $\sigma$ in the sense that 
	$\sigma(x,\pi) = \cF\kappa_x (\pi)$.
\end{lemma}
We will call $\kappa^{\rm Wick}$ the Wick convolution kernel of $\sigma$. Let us denote by $\sigma^{\rm Wick}$ the symbol associated with $\kappa^{\rm Wick}$, $\sigma^{\rm Wick}= \mathcal F \kappa^{\rm Wick}$. Lemma~\ref{lem_kW} can be rephrased as 
$$
	\Op^{\rm Wick} (\sigma)=\Op^{\rm KN} (\sigma^{\rm Wick}).
 $$

\begin{proof}
We first check readily that $\kappa^{\rm Wick}\in \cC_b(G,L^1(G))$ with
$$
\int_G \sup_{x\in G}| \kappa_x^{\rm Wick}(w)| dw
\leq \int_G \sup_{x'\in G}| \kappa_{x'}(w)|  
\left(\int_G |a|(z' w^{-1} )|a|(z')  dz'\right)\,dw
\leq \int_G \sup_{x'\in G}| \kappa_{x'}(w)| dw ,
$$
by the Cauchy-Schwartz inequality since $\|a\|_{L^2(G)}=1$.
\smallskip 

Let us now prove the core of the statement. Let $f\in \cC_c(G)$.  Properties \eqref{eq_cFf1*f2} and~\eqref{eq:B} yield for $(x,\pi)\in G\times\Ghat$  
$$
\sigma \cB[f](x,\pi) = (\widehat \kappa_x \cF (f a(x^{-1}\,\cdot)))(\pi)
=\cF\left ( (f\, a(x^{-1}\cdot ))*\kappa_x\right )(\pi)
$$
 so by equation~\eqref{eq:2.75} of Part 2 of Proposition \ref{prop_B}, we obtain for $x\in G$
\begin{align*}
\Op^{\rm Wick} (\sigma) f (x) 
&=\int_G  (f\, a(z^{-1}\,\cdot ))*\kappa_z(x) \, \bar  a(z^{-1}x) 
dz\\
&=\int_{G\times G} f(y) a(z^{-1}y) \kappa_z(y^{-1} x)\, \bar  a(z^{-1}x) 
dydz,
\end{align*}
and we recognise 
$f*\kappa^{\rm Wick}_x(x).$
\end{proof}

\subsubsection{Some properties of $\Op^{\rm KN}$ and $\Op^{\rm Wick}$}
\label{subsubsec_compOPs}
In our definitions of the quantizations, we  choose to  act on the left by $\tau$ in \eqref{eq_OpKN} or equivalently to place $\kappa_x$ on the right of the convolution product in \eqref{eq_OpKN} 
 while we made choices in the writing of $\sF_{x,\pi}$.
These choices imply that our quantization interact well with the left translations $L_{x_0}$ by $x_0$ on functions, i.e. $L_{x_0}f(x) = f(x_0x)$ for any function $f$ defined on $G$, and also on symbols: $L_{x_0}\sigma (x,\pi) = \sigma(x_0 x,\pi)$. Indeed, 
we check readily that 
$$
\Op^{{\rm Wick}}(L_{x_0} \sigma) 
=
L_{x_0}\Op^{{\rm Wick}
}(\sigma)L_{x_0}^{-1}
\quad\mbox{and}\quad 
\Op^{\rm KN}(L_{x_0}\sigma)=
L_{x_0} \Op^{\rm KN}(\sigma) L_{x_0}^{-1} 
.
$$

The Wick quantization $\Op^{\rm Wick}$ has the advantage of preserving self-adjointness and of being naturally positive.
Indeed, for any $\sigma\in L^\infty (G\times \Ghat)$, we have
$$
(\Op^{\rm Wick}(\sigma))^* = \cB^* \sigma^* \cB=\Op^{\rm Wick}(\sigma^*),
$$
so
if $\sigma$ is self-adjoint in the sense that $\sigma(x,\pi)=\sigma(x,\pi)^*$ for almost all $(x,\pi)\in G\times \Ghat$, then $\Op^{\rm Wick}(\sigma)$ is self-adjoint. 
Moreover, 
 if $\sigma$ is a non-negative symbol in the sense that the operator $\sigma(x,\pi)$ is bounded below by 0 for almost every $(x,\pi)\in G\times \Ghat$, then 
the corresponding operator acting on $L^2(G\times \Ghat)$ is also non-negative so 
\begin{equation}
	\label{positivity}
	(\Op^{\rm Wick}(\sigma)f,f)_{L^2(G)} = 
	(\sigma \cB[f],\cB[f])_{L^2(G\times\Ghat)}\geq 0 
	\qquad \sigma\in L^\infty(G\times\Ghat), \ f\in \cC_c(G),
\end{equation}
and
$\Op^{\rm Wick}(\sigma)$ is a non-negative operator on $L^2(G)$.
\smallskip

In general, the Kohn-Nirenberg quantization $\Op^{\rm KN}$ will not be positive.
However,  weaker properties of positivity  may be recovered in certain cases via  G\r  arding inequalities in pseudodifferential calculi.
The rest of this paper is devoted to showing  G\r  arding inequalities  in the case of graded nilpotent Lie groups and compact Lie groups.

\section{G\r  arding inequality on compact  Lie groups}
\label{sec_Gcomp}

Here, $G$ is a connected compact Lie group. 
Automatically, all the technical assumptions  mentioned in Section \ref{sec_quantization} (locally compact, unimodular, type I, amenable) are satisfied.
In this case, every irreducible representation is finite dimensional, the dual set $\Ghat$ is discrete and the Plancherel measure is known explicitly:
$\mu(\{\pi\})=d_\pi$ is the dimension of $\pi\in \Ghat$, so that we have the Plancherel formula: 
$$
\|f\|_{L^2(G)}^2= \sum_{\pi\in \Ghat} d_\pi \|\widehat f(\pi)\|_{HS(\cH_\pi)}^2.
$$
A symbol is a family $\sigma = \{\sigma(x,\pi) \in \sL(\cH_\pi) \, :\, (x,\pi)\in G\times \Ghat\}$ of finite dimensional linear maps parametrised by $(x,\pi)$, each acting on the (finite dimensional) space of the representation. 
We can define the Fourier transform not only of integrable functions, but also of any distributions.

\subsection{The pseudodifferential calculus}
\label{subsec_pseudoC_compact}
In this section, we set some notations and recall briefly the global symbol classes defined on $G$ together with some properties of the pseudodifferential calculus. 
We refer to 
\cite{FJFA2015} for more details.

\subsubsection{Definitions}

We start with general definitions. 
We fix a basis $X_1,\ldots,X_n$ for the Lie algebra $\fg$ of the group $G$. 
We keep the same notation for the associated left-invariant vector fields on $G$. For a multi-index $\alpha=(\alpha_1,\ldots,\alpha_n)\in \bN_0^n$, we set $X^\alpha=X_1^{\alpha_1}\ldots X_n^{\alpha_n}$.
For $s\in \N_0$, the Sobolev space $H^s(G)$ 
is the space of the functions $f\in L^2(G)$  such that 
\[
\| f\|_{H^s(G)}= \sup_{|\alpha|=s}\| X^\alpha f\|_{L^2(G)} <+\infty.
\]
The Sobolev spaces $H^s(G)$ with $s>0$ are then defined by interpolation and those with $s<0$ by duality. 

We fix a scalar product on $\fg$ that is invariant under the adjoint action.
The Laplace-Beltrami operator is the differential operator
$\cL=-X_1^2-\ldots - X_n^2$
for any orthonormal basis $X_1,\ldots,X_n$ of~$\fg$. 
Identified with an element of the universal enveloping algebra and keeping the same notation for a representation $\pi$ of $G$ and its infinitesimal counterpart for $\fg$, $\pi(\cL)$ is scalar when $\pi$ is irreducible.
$$
\widehat \cL(\pi):=\pi(\cL)=\lambda_\pi \id_{\cH_\pi},
$$
with $\lambda_\pi\geq 0$. In fact, $\lambda_1=0$
when $\pi$ is the trivial representation 1, while 
$\lambda_\pi>0$ when $\pi\neq 1$. 
\smallskip 

Let us now define symbol classes. 
Let $m\in \bR$ and $1\geq \rho \geq \delta\geq 0$. 
A symbol $\sigma$ is in $S_{\rho,\delta}^m(G)$
when  for any multi-indices  $\alpha,\beta$, there exists $C=C(\alpha,\beta)\geq 0$ such that
$$
\| X^\beta \Delta^\alpha \sigma(x,\pi) \|_{\sL(\mathcal H_\pi)} \leq C (1+\lambda_\pi)^{\frac {m-\rho|\alpha|+\delta|\beta|}2},\qquad 
(x,\pi)\in G\times \Ghat.$$
Above, $\Delta^\alpha$ denotes the intrinsic difference operators (see \cite{FJFA2015,FJFA2020} for more details) or 
 the RT-difference operators (see \eqref{eq_RTdiff} below).
This yields the following semi-norm 
$$
\|\sigma\|_{S_{\rho,\delta}^m,a,b}
:=\max_{|\alpha|\leq a, |\beta|\leq b}
\sup_{(x,\pi)\in G\times \Ghat}
(1+\lambda_\pi)^{-\frac {m-\rho|\alpha|+\delta|\beta|}2}\| X^\beta \Delta^\alpha \sigma(x,\pi) \|_{\sL(\mathcal H_\pi)} .
$$
If $(\rho,\delta)=(1,0)$, we simply write 
$S^m(G)=S^m_{1,0}(G)$.

\smallskip

The following theorem sumarises the main property of the classes of operators obtained by the $\Op^{\rm KN}$-quantization of the classes $S_{\rho,\delta}^m(G)$. As mentioned at the beginning of the section, the reader can refer to~\cite{FJFA2015} where proofs are detailed.

\begin{theorem}
\label{thm_calculus_comp}
For each $m\in \bR$, and $1\geq \rho \geq \delta\geq 0$,
equipped with the semi-norms $\|\cdot \|_{S_{\rho,\delta}^m,a,b}$, 
  $S^m_{\rho,\delta}(G)$ becomes a Fr\'echet space.
The space of operators
 $\Psi_{\rho,\delta}^m(G) := \Op^{\rm KN}(S_{\rho,\delta}^m(G))$ inherits this structure of Fr\'echet space. 
If $\delta\neq 1$, 
the classes of operators
 $\Psi_{\rho,\delta}^\infty(G)=\cup_{m\in \bR}\Psi_{\rho,\delta}^m(G) $ 
 is a pseudodifferential calculus in the sense of  
Definition	\ref{def_pseudo-diff_calculus}.
Moreover, we have the following properties:

\begin{enumerate}
\item The calculus $\Psi_{\rho,\delta}^\infty(G)$ acts continuously on the Sobolev spaces $H^s(G)$ in the following sense: 
if $\sigma\in S_{\rho,\delta}^m(G)$ then $\Op^{\rm KN}(\sigma)$ maps $H^s(G)$ to $H^{s-m}(G)$ for any $s\in \bR$.
Furthermore, the map $\sigma\mapsto \Op^{\rm KN}(\sigma)$ is continuous $S^m_{\rho,\delta}(G)\to \sL(H^s(G),H^{s-m}(G))$.
\item For any $\sigma_1\in S_{\rho,\delta}^{m_1}$
and $\sigma_2\in S_{\rho,\delta}^{m_2}$, we have
$$
\Op^{\rm KN}(\sigma_1)\Op^{\rm KN}(\sigma_2)- \Op^{\rm KN}(\sigma_1\sigma_2)
\in 
\Psi_{\rho,\delta}^{m_1+m_2-(\rho-\delta)}(G).
$$
Furthermore, the map $(\sigma_1,\sigma_2)\mapsto \Op^{\rm KN}(\sigma_1)\Op^{\rm KN}(\sigma_2)- \Op^{\rm KN}(\sigma_1\sigma_2)$ is continuous $S_{\rho,\delta}^{m_1}\times S_{\rho,\delta}^{m_2} \to \Psi_{\rho,\delta}^{m_1+m_2-(\rho-\delta)}(G)$.

\item For any $\sigma\in S_{\rho,\delta}^{m}$, we have
$$
\Op^{\rm KN}(\sigma)^*- \Op^{\rm KN}(\sigma^*)
\in 
\Psi_{\rho,\delta}^{m-(\rho-\delta)}(G).
$$
Furthermore, the map $\sigma\mapsto \Op^{\rm KN}(\sigma)^*- \Op^{\rm KN}(\sigma^*)$ is continuous $S_{\rho,\delta}^{m} \to \Psi_{\rho,\delta}^{m-(\rho-\delta)}(G)$.
\end{enumerate}
 \end{theorem}
 
When $\rho>\delta$ and $\rho\geq 1-\delta$,  this calculus coincides with  the H\"ormander pseudodifferential calculus defined locally via charts.

 \subsubsection{Properties of pseudodifferential operators}

Any $\sigma\in S_{\rho,\delta}^m(G)$ admits a distributional  convolution kernel $\kappa: x\mapsto( z\mapsto \kappa_x(z)) \in \cC^\infty(G,\cD'(G))$, i.e. $\sigma(x,\pi) = \widehat \kappa_x(\pi)$ and  
$$
\Op^{\rm KN}(\sigma)f(x) =f*\kappa_x(x), \quad f\in \cD(G), \ x\in G.
$$

In the following, we will use properties of symbols with respect to the RT-difference operators.
Let us recall that the RT-difference operator $\Delta_q$ associated to $q\in \cC^\infty(G)$ is defined via:
\begin{equation}
	\label{eq_RTdiff}
	\Delta_q \widehat \kappa = \cF (q\kappa), \qquad \kappa\in \cD'(G).
\end{equation}
The following property of RT-difference operators follows readily from \cite[Section 5]{FJFA2015}:
\begin{lemma}
\label{lem_contDelta_q}
Let $m\in \bR$ and $1\geq \rho\geq \delta\geq 0$. 
\begin{enumerate}
	\item If $q\in \cD(G)$, then the map $\sigma \mapsto \Delta_q \sigma$ is continuous  $S_{\rho,\delta}^m(G)\to S_{\rho,\delta}^m(G)$. 
Moreover, $q\mapsto \Delta_q$ is continuous $\cD(G)\to \sL(S_{\rho,\delta}^m(G))$. 
	\item The map $\sigma \mapsto \Delta_{q-q(e_G)} \sigma$ is continuous  $S_{\rho,\delta}^m(G)\to S_{\rho,\delta}^{m-(\rho-\delta)}(G)$ for any $m\in \bR$. 
Moreover, $q\mapsto \Delta_{q-q(e_G)}$ is continuous $\cD(G)\to \sL(S_{\rho,\delta}^m(G), S_{\rho,\delta}^{m-(\rho-\delta)}(G))$. 
\end{enumerate}		
\end{lemma}

Secondly, we will use the following property of right translations:
\begin{lemma}
\label{lem_translationSm}
Let $m\in \bR$ and $1\geq \rho\geq \delta\geq 0$.
If $x_0\in G$, then for any $\sigma\in S_{\rho,\delta}^m(G)$, the symbol $R_{x_0}\sigma  = \{\sigma(xx_0,\pi) : (x,\pi)\in G\times \Ghat\}$ is in $S_{\rho,\delta}^m(G)$.
	Moreover, 
	 	the map
 $(x_0,\sigma) \mapsto R_{x_0} \sigma $ is  continuous $G\times S_{\rho,\delta}^m(G)\to S_{\rho,\delta}^{m} (G)$. 
\end{lemma}

\begin{proof}
We consider the semi-norms 
$$
\|\sigma\|'_{S_{\rho,\delta}^m,a,b}
:=\max_{|\alpha|\leq a, |\beta|\leq b}
\sup_{(x,\pi)\in G\times \Ghat}
(1+\lambda_\pi)^{-\frac {m-\rho|\alpha|+\delta|\beta|}2}\| \tilde X^\beta \Delta^\alpha \sigma(x,\pi) \|_{\sL(\mathcal H_\pi)} ,
$$
  where we have used the right-invariant derivatives $\tilde X^\beta$ instead of the left-invariant ones $X^\beta$.
 As $G$ is compact, the semi-norms $\|\cdot\|'_{S_{\rho,\delta}^m,a,b}$ generate the topology of $S^m_{\rho,\delta}(G)$.
 We observe that $\|R_{x_0} \sigma\|'_{S_{\rho,\delta}^m,a,b} = \|\sigma\|'_{S_{\rho,\delta}^m,a,b}$.
 This implies the statement.
\end{proof}

Finally, we will need some properties of convolution in the $x$-variable of a symbol. They are summarised in the next statement, but first let us define what we mean by convolution of a symbol. If $\sigma\in S_{\rho,\delta}^m(G)$ and $\varphi\in \cC^\infty(G)$, then we denote by 
 $\sigma \,*\, \varphi $ the symbol
\begin{align*}
&\sigma\,*\, \varphi = \{\sigma\,*\, \varphi \,(x,\pi)\, : \, (x,\pi)\in G\times \Ghat\}, \\
\mbox{with}\qquad &
  \sigma \,*\, \varphi\,(x,\pi)  = \int_G  \sigma(z,\pi) \varphi( z^{-1}x)dz= \int_G  \sigma(x(z')^{-1},\pi) \varphi(z')dz', \;\; (x,\pi)\in G\times \Ghat.
\end{align*}

\begin{lemma}
\label{lem_convSm}
Let $m\in \bR$ and $1\geq \rho\geq \delta\geq 0$.
\begin{enumerate}
 \item If $\sigma\in S_{\rho,\delta}^m(G)$ and $\varphi\in \cC^\infty(G)$, then 
	 we have $\sigma \,*\, \varphi \in S_{\rho,\delta}^{m}(G)$ with for any semi-norm $\|\cdot\|_{S_{\rho,\delta}^m,a,b}  $
	 $$
	\|\sigma \,*\, \varphi  \|_{S_{\rho,\delta}^m,a,b} 
	\leq C \|\sigma\|_{S_{\rho,\delta}^m,a,0},
	$$
 where $C$ is a constant depending on $\varphi$ and $b$.
 This implies that $\sigma\mapsto \sigma*\varphi$ is continuous on  $S_{\rho,\delta}^m(G)$.
 \item 
 Furthermore, if $\int_G \varphi(y)dy =1$ then we have for any semi-norm $\|\cdot\|_{S_{\rho,\delta}^{m+\delta},a,b}  $
 $$
\|\sigma*\varphi  - \sigma\|
_{S^{m+\delta }_{\rho,\delta},a,b}
\leq 
C' \int_G |y^{-1}||\varphi(y)|dy \ 
\|\sigma\|_{S^{m}_{\rho,\delta},a,b+1},$$
 where $C'$ is a constant depending on $b$, and where $|z|$ denotes the Riemann distance of $z\in G$ to the neutral element $e_G$.
 (The invariant Riemannian distance is induced by our choice of scalar product on $\fg$.)
\end{enumerate}
\end{lemma}
\begin{proof} 
We observe that 
$$
\Delta^\alpha X^\beta (\sigma*\varphi)(x,\pi) 
= \Delta^\alpha \sigma*X^\beta \varphi \, (x,\pi), \quad\mbox{and}\quad
\|\sigma * \varphi(x,\pi)\|_{\sL(\cH_\pi)}\leq 
\|\varphi\|_{L^1(G)}\|\sigma\|_{L^
\infty (G\times \Ghat)}.
$$
This readily implies Part (1).

Assume $\int_G \varphi(y)dy =1$.
Part (2) will follow from 
the Taylor estimate:
$$
\forall f\in \cC^1(G),\qquad 
|f(z)-f(0)| \lesssim_G |z| \max_{j=1,\ldots,n}\sup_{z'\in G} | X_j f(z')|,
$$
Indeed, we may write:
$$
\Delta^\alpha X^\beta (\sigma*\varphi -\sigma)(x,\pi)
= \int_G (\Delta^\alpha X^\beta_x\sigma(x y^{-1},\pi) -\Delta^\alpha X^\beta\sigma(x,\pi)) \varphi (y)dy 
$$
so 
\begin{align*}
    \|\Delta^\alpha X^\beta (\sigma*\varphi -\sigma)(x,\pi)\|_{\cH_\pi}
&\leq  \int_G \| \Delta^\alpha X^\beta_x\sigma(x y^{-1},\pi) -\Delta^\alpha X^\beta\sigma(x,\pi))\|_{\cH_\pi}  
|\varphi (y)
| dy \\
&\lesssim_G  
\max_{j=1,\ldots,n}\sup_{z\in G}\|  \Delta^\alpha  X_{j,z} X^\beta_x\sigma(xz,\pi))\|_{\cH_\pi} \int_G |y^{-1}  |\
|\varphi (y)| dy.
\end{align*}
We conclude with 
$$
\max_{j=1,\ldots,n}\sup_{z\in G}\|  \Delta^\alpha  X_{j,z} X^\beta_x\sigma(xz,\pi))\|_{\cH_\pi}  
\lesssim_{G, |\beta|} \max_{|\beta'|=|\beta|+1}\sup_{x'\in G}\|  \Delta^\alpha  X^{\beta'}\sigma(x',\pi))\|_{\cH_\pi}  .
$$
\end{proof}

\subsection{Proof of the G\aa rding inequality}
\label{subsec_pfGarding_comp}
Here, we prove the following $(\rho,\delta)$-generalisation of Theorem \ref{thm_Gardingcomp}:
\begin{theorem}
\label{thm_Gardingcomp_rhodelta}
Let $G$ be a connected compact Lie group. 
Let $m\in \bR$ and $1\geq \rho>\delta\geq 0$.
	Assume that the symbol $\sigma\in S_{\rho,\delta}^{m}(G)$ satisfies the positivity condition $\sigma \geq 0$.
	Then, for all $\eta>0$,  there exists a constant $C_\eta >0$ such that 
	\begin{equation}\label{eq:w_g}
	\forall f\in \cC^\infty(G), \qquad 
	\Re \left (\Op^{\rm KN}(\sigma) f, f\right )_{L^2(G)} \geq  -\eta\,   \,  \| f\| _{H^{\frac{m+\delta}2}(G)}^2-C_\eta \|f\|_{H^{\frac{m-(\rho-\delta)}2}(G)}^2 .
\end{equation}	 
Moreover, if $\delta=0$ and  $\sigma\in S_{\rho,\delta}^{m}(G)$ satisfies the ellipticity condition $\sigma \geq 
	c_0 (\id  +\widehat \cL)^{m/2}$, 
that is,	
	$$
	\sigma(x,\pi) \geq c_0 (1+\lambda_\pi)^{\frac m2} \id_{\cH_\pi}, \qquad (x,\pi)\in G\times \Ghat,
	$$
	for some constant $c_0>0$.
	Then,   there exist  constants $c,C>0$ such that 
	$$
	\forall f\in \cC^\infty(G), \qquad 
	\Re \left (\Op^{\rm KN}(\sigma) f, f\right )_{L^2(G)} \geq  c\| f\|_{H^{\frac m2}(G)}^{2}-C \|f\|_{H^{\frac{m-\rho}2}(G)}^2 .
$$	 
\end{theorem}

\begin{remark}\label{rem:pluto}
    \begin{enumerate}
\item When $\delta=0$ and $\rho=1$, the second part of Theorem~\ref{thm_Gardingcomp_rhodelta} is Theorem \ref{thm_Gardingcomp}. 

\item When $\delta\not=0$, the inequality~\eqref{eq:w_g} differs from the usual sharp G\aa rding inequality in which the term with coefficient $\eta$ does not appear. The inequality~\eqref{eq:w_g} differs from the  straightforward estimate 
\[ \Re \left (\Op^{\rm KN}(\sigma) f, f\right )_{L^2(G)} \geq  - C_1\| f\|_{H^{\frac m2}(G)}^{2}\]
by the fact that $\eta$ can be chosen as small as possible. This term shows the limit of our approach in the $(\rho,\delta)$-calculus when $\delta\not=0$.
\end{enumerate}
\end{remark}

\subsubsection{The main ingredients of the proof}

The main ingredients for our proof of Theorem \ref{thm_Gardingcomp_rhodelta}
are
firstly an analysis of the Wick quantization in the $\Psi^\infty_{\rho,\delta}$-calculus, and secondly the choice of $a$ in the Wick quantization.
\smallskip 

We observe that if $\sigma\in S_{\rho,\delta}^m(G)$ with $m\leq 0$, then 
 $\sigma\in L^\infty(G\times \Ghat)$ and we can consider $\Op^{\rm Wick}(\sigma)$ and study it in the   $\Psi^\infty_{\rho,\delta}$-calculus.
 
\begin{lemma}
\label{lem_OpW}
Here, we consider the Wick quantization $\Op^{{\rm Wick},a}$ with a smooth function 
$a:G\to \bC$. 
Let $m\in \bR$ and $1\geq \rho\geq \delta\geq 0$ with $\delta\neq 1$.
\begin{enumerate}
	\item If $\sigma\in S_{\rho,\delta}^m(G)$ with $m\leq 0$, then 
 $\sigma\in L^\infty(G\times \Ghat)$ and $\Op^{\rm Wick}(\sigma)\in \Psi_{\rho,\delta}^m(G)$. Moreover, the map $\sigma \mapsto \Op^{\rm Wick}(\sigma)$ is continuous $S_{\rho,\delta}^m(G)\to \Psi_{\rho,\delta}^m(G)$.

 \item  If $\sigma\in S_{\rho,\delta}^m(G)$ with $m\leq 0$, 
	 then we have
	$$
	\Op^{\rm Wick}(\sigma) - \Op^{\rm KN} (\sigma *|a|^2) \in \Psi_{\rho,\delta}^{m-(\rho-\delta)}(G).
	$$ 
Moreover, the map $\sigma \mapsto \Op^{\rm Wick}(\sigma) - \Op^{\rm KN} (\sigma *|a|^2)$ is continuous $S_{\rho,\delta}^m(G)\to \Psi_{\rho,\delta}^{m-(\rho-\delta)}(G)$.
\end{enumerate}
\end{lemma}

\begin{proof}[Proof of Lemma \ref{lem_OpW}]
We may rephrase Lemma \ref{lem_kW} as $\Op^{\rm Wick}(\sigma) = \Op^{\rm KN}(\sigma^{\rm Wick})$
with 
$$
\sigma^{\rm Wick} (x,\pi)= \int_G \Delta_{q_z} \sigma (xz^{-1},\pi) dz
$$
where $q_z(w) = a(zw^{-1})\bar a(z)$.
By Lemmata \ref{lem_translationSm} and \ref{lem_contDelta_q} (1), this implies Part 1.

We observe that 
$$
\int_G \Delta_{q_z(e_G)} \sigma (xz^{-1},\pi)\, dz
=\int_G  |a|^{2}(z) \, \sigma(xz^{-1},\pi)\, dz
= \sigma*|a|^2 (x,\pi).
$$
Hence, Lemma \ref{lem_contDelta_q} (2) implies Part 2.
\end{proof}

 The choice of the functions
 $a$ for the Wick quantization is at the core of our proof of the G\aa rding inequality. We do it  in relation to an approximation of the identity.
By an approximation of the identity on a compact Lie group $G$, we mean here a family of functions  $\varphi_t\in \cD(G)$, $t>0$, satisfying 
$\int_G \varphi_t(z)dz =1$ for any $t>0$ 
and for any neighbourhood $V$ of the neutral element $e_G$, $\lim_{t\to 0} \int_{z\notin V} |\varphi_t(z)| dz=0$ and $\sup_{t\in (0,1]} \int_{z\in V} |\varphi_t(z)| dz<\infty $.
We then have  
\begin{equation}
\label{eq_convphitLinfty}
  \forall \psi\in\cC(G,\bC),\qquad \lim_{t\rightarrow0}\max_{x\in G}|\psi(x)-\psi*\varphi_t (x)|=0.  
\end{equation}
The  properties regarding the approximation of the identity that we will use in our proof of G\aa rding inequalities below are  summarised in the following lemma:

\begin{lemma}
  \label{lem_control}
  Let
  $m\in \bR$ and $1\geq \rho\geq \delta\geq 0$.
  Let $\varphi_t$, $t>0$, be an approximation of the identity on the compact Lie group $G$ (as defined above). 
  We assume that it satisfies $\varphi_t(z)\geq 0$ for any $z\in G$ and $t\in (0,1)$. Then 
  for any $\sigma\in S^{m}_{\rho,\delta}(G)$, as $t\to 0$,
  $\sigma*\varphi_t$ converges to $\sigma$ in $S^{m+\delta}_{\rho,\delta}(G)$, that is, for any semi-norm $\|\cdot\|_{S^{m+\delta}_{\rho,\delta}(G),a_0,b_0}$ we have
$$
\lim_{t\to 0} \|\sigma -  \sigma*\varphi_t\|_{S^{m+\delta}_{\rho,\delta}(G),a_0,b_0}
=0.
$$
\end{lemma}

\begin{proof}[Proof of Lemma \ref{lem_control}]
We observe that for any $\varphi\in \cC^\infty(G)$, 
we have
$$
\|\sigma*\varphi  - \sigma\|
_{S^{m+\delta }_{\rho,\delta},a_0,b_0}
\lesssim_G \int_G |y^{-1}||\varphi(y)|dy\  
\|\sigma\|_{S^{m}_{\rho,\delta},a_0,b_0+1},$$
by Lemma \ref{lem_convSm} (2). 
By \eqref{eq_convphitLinfty} and because $\varphi_t\geq 0$, we have
$$
\lim_{t\to 0}\int_G |y^{-1}||\varphi_t(y)|dy
=\lim_{t\to 0}\int_G |y^{-1}|\varphi_t(y)dy
=\lim_{t\to 0}
(|\cdot|*\varphi_t )(0)
= |0|=0.
$$
Therefore $\lim_{t\to 0} \|\sigma -  \sigma*\varphi_t\|_{S^{m+\delta}_{\rho,\delta}(G),a_0,b_0}=0$.
\end{proof}

In the proof of the G\aa rding inequality for symbol of order 0 below, 
we will choose an approximation of the identity $\varphi_t$, $t>0$, that never vanishes, i.e.  $\varphi_t(x)>0$ for any $x\in G$ and $t>0$, and then take $a:=\sqrt{\varphi_t}$.
Such an approximation of the identity $\varphi_t$ is obtained by considering the heat kernel $p_t$ \cite{varo,FJFA2015}, that is, the convolution kernel of $e^{-t \cL}$.

\subsubsection{Proof of Theorem \ref{thm_Gardingcomp_rhodelta} }

We start by proving the result for $m=0$, then we extend the result to any $m\in\R$.

\begin{proof}[Proof of Theorem \ref{thm_Gardingcomp_rhodelta} for $m=0$] 
Let
 $\sigma\in S_{\rho,\delta}^0(G)$ 
satisfying $\sigma(x,\pi)=\sigma(x,\pi)^*\geq 0$ for any $(x,\pi)\in G\times \Ghat$.
The link between  the Wick and Kohn-Nirenberg quantizations (Lemma \ref{lem_OpW}) and the properties  of the pseudodifferential calculus (Theorem \ref{thm_calculus_comp}) imply
$$
\frac 12(\Op^{\rm KN}(\sigma)+\Op^{\rm KN}(\sigma)^*) 
= \Op^{\rm Wick} (\sigma) + \Op^{\rm KN} (\sigma - \sigma*|a|^2 ) + \Op^{\rm KN}(\tau),
$$
with $\tau\in S_{\rho,\delta}^{-(\rho-\delta)}(G)$. Hence, we obtain:
\begin{align}\label{dalton}
	\Re (\Op^{\rm KN}(\sigma) f, f)_{L^2(G)}
&\geq  (\Op^{\rm Wick}(\sigma) f, f)_{L^2(G)} 
-\left| \left(\Op^{\rm KN} (\sigma -  \sigma*|a|^2)  f,f \right)_{L^2(G)}\right|
\\
\nonumber &\qquad -  \|\Op^{\rm KN}(\tau)\|_{\sL(H^{-\frac {\rho-\delta}2}(G),H^{\frac {\rho-\delta}2}(G))}
\|f\|^2_{H^{-\frac {\rho-\delta}2}(G)}.
\end{align}
The property \eqref{positivity} of the Wick quantization and the hypothesis 
 $\sigma(x,\pi)\geq 0$ for any $(x,\pi) \in G\times \Ghat$ yield
$$
	(\Op^{\rm Wick}(\sigma) f, f)_{L^2(G)} 
	\geq 0.
$$
We then choose $a:=\sqrt{\varphi_t}$ with $\varphi_t(x)>0$ an approximation of the identity.
By Theorem \ref{thm_calculus_comp} (1), we have
\begin{align*}
    \left|\left(\Op^{\rm KN} (\sigma -  \sigma *|a|^2 )f,f\right)_{L^2(G)} \right|
    \leq
    \|\Op^{\rm KN} (\sigma -  \sigma *|a|^2 )\|_{\sL(H^{\frac\delta 2},H^{-\frac\delta 2})} \|f\|_{H^{\frac{\delta} 2}(G)}^{2}
    \leq \eta \|f\|_{H^{\frac{\delta} 2}(G)}^{2}
\end{align*}
by Lemma~\ref{lem_control} for some $t=t(\eta)$.
Finally, 
 the properties of the pseudodifferential calculus (Theorem~\ref{thm_calculus_comp}) imply that the operator norm 
$
\|\Op^{\rm KN}(\tau)\|_{\sL(H^{-\frac {\rho-\delta}2}(G)),H^{\frac {\rho-\delta}2}(G))}$
is finite. 
Therefore, collecting in these facts, we deduce from~\eqref{dalton}
\begin{equation*}
	\Re (\Op^{\rm KN}(\sigma) f, f)_{L^2(G)}
\geq  (\Op^{\rm Wick}(\sigma) f, f)_{L^2(G)} 
-\eta \| f\|^2 _{H^{\frac{\delta}{2}}(G)}
-C_\eta \| f\|^2 _{H^{-\frac {\rho-\delta}2}(G)},
\end{equation*}
for some constant $C_\eta=\|\Op^{\rm KN}(\tau)\|_{\sL(H^{-\frac {\rho-\delta}2}(G),H^{\frac {\rho-\delta}2}(G))}$ (note that $\tau$ depends on~$a=\sqrt{\varphi_t}$, or, equivalently on~$t$, and thus, on~$\eta$).
This concludes the first part of the Theorem, in the case $\sigma\in S^0_{\rho,\delta}(G)$.
\smallskip 

The second part of the Theorem, with $\delta=0$, is obtained in a similar manner by observing that the property \eqref{positivity} of the Wick quantization and the hypothesis 
 $\sigma(x,\pi)\geq c_0$ for any $(x,\pi) \in G\times \Ghat$ yield
$$
	(\Op^{\rm Wick}(\sigma) f, f)_{L^2(G)} 
	\geq 
	c_0 \|f\|^2_{L^2(G)}.
$$
One then choose $\eta < c_0$ so that the term $ \eta \|f\|_{H^{\frac \delta 2}(G)}^{2}= \eta \|f\|_{L^2(G)}^{2}$ is absorbed by $c_0 \|f\|^2_{L^2(G)}$. 
\end{proof}

\begin{proof}[Proof of Theorem \ref{thm_Gardingcomp_rhodelta} for any $m\in \bR$] 
Let $\sigma\in S_{\rho,\delta}^{m}(G)$ be such that $\sigma \geq 0$.
By the properties of the pseudodifferential calculus, we may write
$$
(\id+\cL)^{-m/4} \Op^{\rm KN}(\sigma) \,(\id+\cL)^{-m/4} =
\Op^{\rm KN}(\sigma_1)+
\Op^{\rm KN}(\tau_1),
$$
with $\sigma_1\in S_{\rho,\delta}^0(G)$ given by $\sigma_1(x,\pi):= (1+\lambda_\pi)^{-m/2} \sigma (x,\pi)$,  and 
$\tau_1\in S_{\rho,\delta}^{-(\rho-\delta)}(G)$.
We observe that $\sigma_1$ satisfies the hypothesis of Theorem \ref{thm_Gardingcomp_rhodelta} with $m=0$.
Therefore for
 any $f\in \cC^\infty(G)$, 
setting $f_1=(\id+\cL)^{m/4} f$, we have
\begin{align*}
\Re (\Op^{\rm KN}(\sigma) f, f)_{L^2(G)}
&=
\Re ((\id+\cL)^{-\frac{m}4} \Op^{\rm KN}(\sigma) \,(\id+\cL)^{-\frac{m}4}
 f_1, f_1)_{L^2(G)}	\\
 &= \Re (\Op^{\rm KN}(\sigma_1) f_1, f_1)_{L^2(G)}
 + \Re (\Op^{\rm KN}(\tau_1) f_1, f_1)_{L^2(G)}\\
 &\geq -\eta \| f_1\|_{ H^{\frac\delta 2}(G)}^{2} -C _\eta \|f_1\|_{H^{-\frac {\rho-\delta}2}(G)}^{2},
\end{align*}
by Theorem \ref{thm_Gardingcomp_rhodelta} with $m=0$ applied to $\sigma_1$ and the properties of the pseudodifferential calculus applied to $\tau_1$. 
The conclusion follows from 
$\|f_1\|_{H^{-\frac {\rho-\delta}2}(G)} \sim \|f\|_{H^{\frac{m -(\rho-\delta)}2}(G)}$.
\end{proof}

\section{G\r  arding inequality on graded nilpotent  Lie groups}
\label{sec_Gnilp}

Here, we prove the G\r  arding inequality on a graded nilpotent  Lie group $G$. 
Before this,  we recall some definitions and notation about this class of groups and the associated pseudodifferential calculus. We refer to~\cite{FR} for more details.

\subsection{Preliminaries on graded nilpotent groups}

A graded group $G$ is a connected simply connected nilpotent Lie group whose (finite dimensional, real) Lie algebra $\mathfrak g$ admits an $\N$-gradation into linear subspaces, 
$$
\mathfrak g = \oplus_{j=1}^\infty \mathfrak g_j 
\quad\mbox{with} \quad [\mathfrak g_i,\mathfrak g_j]\subseteq \mathfrak g_{i+j}, \;\; 1\leq i\leq j,
$$
where all but a finite number of subspaces $\mathfrak g_j $ are trivial.
We denote by $r=r_G$ the smallest integer~$j$ such that all the subspaces $\mathfrak g_j $, $j>r$,  are trivial.  
If the first stratum $\mathfrak g_1$ generates  the whole Lie algebra, then $\mathfrak g_{j+1}= [\mathfrak g_1,\mathfrak g_j]$ for all $j\in\N_0$  and $r$ is the step of the group; the group $G$ is then said to be stratified, and also  (after a choice of basis or inner product for $\mathfrak g_1$)  Carnot.

\subsubsection{The exponential map and functional spaces}
 
 The product law on $G$ is derived from the exponential map  $\exp_G : \mathfrak g \to G$ 
which  is a global diffeomorphism from $\mathfrak g$ onto $G$. Once a basis $X_1,\ldots,X_{n}$
for~$\mathfrak g$ has been chosen, we may  identify 
the points $(x_{1},\ldots,x_n)\in \mathbb R^n$ 
 with the points  $x=\exp(x_{1}X_1+\cdots+x_n X_n)$ in~$G$, $n={\rm dim} \,\mathfrak g$.
It allows us to define 
the (topological vector) spaces $\mathcal C^\infty(G)$, $\cD(G)$ and $\mathcal S(G)$  of smooth, continuous and compactly supported, and  Schwartz functions on $G$ identified with $\mathbb R^n$; 
note that the resulting spaces are intrinsically defined as spaces of functions on~$G$ and do not depend on a choice of basis. 

The exponential map induces a Haar measure $dx$ on $G$ which is invariant under left and right translations and  defines Lebesgue spaces on~$G$.

Finally, it is worth mentioning that in the present case of a graded group $G$, the dual set $\Ghat$ and the Plancherel measure $\mu$ can be explicitly described via
    Kirillov's orbit method~\cite{CG90}.

\subsubsection{Adapted basis and dilations}\label{sec:adap_basis}
We now construct a basis adapted to the gradation. 
Set $d_j={\rm dim} \, \mathfrak g_j $ for $1\leq j\leq r$.
We choose a basis $\{X_1,\ldots, X_{d_1}\}$ of $\mathfrak g_1$ (this basis is possibly reduced to $\emptyset$), then 
$\{X_{d_1+1},\ldots,  X_{d_1+d_2}\}$ a basis of $\mathfrak g_2$
(possibly $\{0\}$)
and so on.
Such a basis $\mathcal B=(X_1, \cdots , X_{d_1+\cdots +d_{r}})$
of $\mathfrak g$ is said to be adapted to the gradation; and we have $n=d_1+\cdots + d_r$.

\smallskip 

The Lie algebra 
 $\mathfrak g$ is  a homogeneous Lie algebra equipped with 
the family of dilations  $\{\delta_r, r>0\}$,  $\delta_r:\mathfrak g\to \mathfrak g$, defined by
$\delta_r X=r^\ell X$ for every $X\in \mathfrak g_\ell$, $\ell\in \N$
\cite{folland+stein_82,FR}.
We re-write the set of integers $\ell\in \N$ such that $\mathfrak g_\ell\not=\{0\}$
into the increasing sequence of positive integers
 $\upsilon_1,\ldots,\upsilon_n$ counted with multiplicity,
 the multiplicity of $\mathfrak g_\ell$ being its dimension.
 In this way, the integers $\upsilon_1,\ldots, \upsilon_n$ become 
 the weights of the dilations and we have $\delta_r X_j =r^{\upsilon_j} X_j$, $j=1,\ldots, n$,
 on the chosen basis of $\mathfrak g$.
 The associated group dilations are defined by
$$
\delta_r(x)=
rx
:=(r^{\upsilon_1} x_{1},r^{\upsilon_2}x_{2},\ldots,r^{\upsilon_n}x_{n}),
\quad x=(x_{1},\ldots,x_n)\in G, \ r>0.
$$
In a canonical way,  this leads to the notions of homogeneity for functions and operators. 
For instance, the Haar measure is homogeneous of degree
$$
Q:=\sum_\ell \upsilon_\ell \dim \fg_\ell,
$$
which is called the homogeneous dimension of the group. 
Another example is the vector field corresponding to an element $X\in \fg_\ell$: it  is $\ell$-homogeneous. 

An important class of homogeneous map are the homogeneneous quasi-norms, 
that is, a $1$-homogeneous non-negative map $G \ni x\mapsto |x|$ which is symmetric and definite in the sense that $|x^{-1}|=|x|$ and $|x|=0\Longleftrightarrow x=0$.
In fact, all the homogeneous quasi-norms are equivalent in the sense that if $|\cdot|_1$ and $|\cdot|_2$ are two of them, then 
$$
\exists C>0,\qquad \forall x\in G,
\qquad C^{-1} |x|_1 \leq |x|_2 \leq C |x|_1.
$$
Examples may be constructed easily, such as  
$$
|x| = (\sum_{j=1}^n |x_j|^{N/\upsilon_j})^{ 1/N} \ \mbox{for any}\ N>0,
$$
 with the convention above.

In the rest of the paper, 
we assume that we have fixed a basis $X_1,\ldots, X_n$ of $\fg$ adapted to the gradation.
We keep the same notation for the associated left-invariant vector fields on $G$, 
and we denote the corresponding right invariant vector fields by $\tilde X_1, \ldots,\tilde X_n$. For a multi-index $\alpha=(\alpha_1,\ldots,\alpha_n)\in \bN_0^n$, we set $X^\alpha=X_1^{\alpha_1}\ldots X_n^{\alpha_n}$
and $\tilde X^\alpha=\tilde X_1^{\alpha_1}\ldots \tilde  X_n^{\alpha_n}$.
The differential operators $X^\alpha$ and $\tilde X^\alpha$ are homogeneous of degree 
$$
[\alpha] = \upsilon_1\alpha_1+\ldots+\upsilon_n \alpha_n.
$$
Left and right vector fields and translations have many relations.
For instance,  for any function $f\in \cC^\infty(G)$ and $x,x_0\in G$, we have
$X_{j,x} f(xx_0) = \tilde X_{j,x_0}f(xx_0)$. Since 
the left and right differential operators are related by polynomial relations (see \cite[Corollary 3.1.30]{FR}), this implies 
\begin{equation}
\label{eq_Xalphafxx0}
  X_x^\alpha f(xx_0)
=
\tilde X_{x_0}^\alpha f(xx_0)
=
\sum_{[\beta]=[\alpha]}P_{\alpha,\beta}(x_0) X^\beta f(xx_0),  
\end{equation}
where the $P_{\alpha,\beta}$'s are $[\beta]-[\alpha]$-homogeneous polynomials.

For $s\in \bN_0$ a common multiple of the dilation's weights $\upsilon_1,\ldots,\upsilon_n$, the Sobolev spaces $L^2_s(G)$ is defined as the set of functions $f\in L^2(G)$ for which
\[
\| f\|_{L^2_s(G)} = \sup_{[\alpha]=s} \| X^\alpha f\|_{L^2(G)}<+\infty.
\]
For other values of $s>0$, they are obtained by interpolation, and by duality for $s<0$. These are well defined Banach spaces, \cite[Section 4.4]{FR}.

\subsubsection{Approximation of identity}
Below, 
we will use approximations of the identity built using dilations in the following sense:
\begin{lemma} \label{joe}\cite{folland+stein_82,FR}
    Let $\varphi_1\in \cS(G)$ with $\int_G \varphi_1=1$.
Consider the family of integrable functions 
$\varphi_t  = t^{-Q}\varphi_1\circ \delta_{t^{-1}}$, $t>0$.
	 The family of functions 
$\varphi_t$, $t>0$,  form an approximation of identity on $L^p(G)$, $p\in [1,\infty)$ and on the space $\cC_0(G)$ of continuous functions vanishing at infinity in the sense that
$$
\lim_{t\to 0} \|\psi-\psi*\varphi_t \|_{L^p(G)}=0,
$$ 
for $\psi\in L^p(G)$ (resp.  $\cC_0(G)$) for $p\in [1,\infty)$ (resp. $p=\infty$).
\end{lemma}

This is  more convenient than  considering only heat kernels $p_t$ of a positive Rockland operators~$\cR$, i.e. 
the convolution kernels of $e^{-t\cR}$, $t>0$, although the latter do provide approximations of the identity \cite{FR,folland+stein_82}. 
When $G$ is stratified and $\cR$ is a sub-Laplacian, the heat kernels will be non-negative and never vanishing.  However, these properties of the heat kernel for a general positive Rockland operator are not guaranteed in the graded case, and we observe that the heat kernel being positive and never vanishing was used in the proof of the compact case in Section \ref{subsec_pfGarding_comp}.
Furthermore, in our proof in the nilpotent case below, considering approximations of the identity built using dilations is, in fact, more practical.

\smallskip

We discuss in the next section the pseudodifferential calculus on nilpotent graded groups and its properties.

\subsection{The pseudodifferential calculus}
\label{subsec_pseudoC_graded}

In this section, we set some notation and recall briefly  the global symbol classes defined on $G$ together with some properties of the pseudodifferential calculus. 
We refer to 
\cite{FR} for more details.

\subsubsection{The symbol classes and the calculus}

A symbol $\sigma$ is in $S_{\rho,\delta}^m(G)$
when  for any multi-indices  $\alpha,\beta\in \bN_0^n$ and $\gamma\in \bR$, there exists $C=C(\alpha,\beta)$ such that we have for almost $(x,\pi)\in G\times \Ghat$,
\begin{equation}
	\label{eq_SmnormGnilp}
	\|\pi(\id+\cR)^{-\frac {m-\rho[\alpha] + \delta[\beta]+\gamma}{\nu}} X^\beta \Delta^\alpha \sigma(x,\pi) \pi(\id+\cR)^{\frac {\gamma}{\nu}}\|_{\sL(\mathcal H_\pi)} \leq C_{\alpha,\beta,\gamma},
\end{equation}
where $\cR$ is a (and then any) positive Rockland operator of homogeneous degree $\nu$; 
 we may assume $\gamma\in \bZ$.  

 In~\eqref{eq_SmnormGnilp}, the difference operator $\Delta^\alpha$ is the difference operator $\Delta_{x^\alpha}$ for the monomial $x^\alpha$  in the coordinates $x_j$. 
Generalising the definition in the compact setting,  the difference operator $\Delta_q$ associated to $q\in \cC^\infty(G)$ is defined via
$\Delta_q \widehat \kappa = \cF (q\kappa)$ for any $\kappa\in \cS'(G)$ for which $\kappa$ and $q\kappa$ admits a Fourier transform (see \cite{FR}).

We set 
\[
\|\sigma\|_{S^m_{\rho,\delta},a,b,c} := \max_{|\alpha|\leq a, |\beta|\leq b, |\gamma|\leq c} C_{\alpha,\beta,\gamma}
\]
for the best constants $C_{\alpha,\beta,\gamma}$ in \eqref{eq_SmnormGnilp} and $a,b,c\in \bN_0$. If $(\rho,\delta)=(1,0)$, we simply write 
$S^m(G)=S^m_{1,0}(G)$.
\smallskip 

The following theorem summarises the main properties of the classes of operators obtained by the $\Op^{\rm KN}$-quantization of the classes $S^m_{\rho,\delta}(G)$;
the  Sobolev spaces $L^2_s(G)$ adapted to the graded nilpotent Lie group $G$ were studied in \cite{FR,FRSob} generalising slightly the stratified case \cite{folland75}.

\begin{theorem}
\label{thm_calculus_Gnilp}
Theorem \ref{thm_calculus_comp} holds for $G$ a graded nilpotent Lie group when replacing the symbol classes with the ones defined  above and the Sobolev spaces with $L^2_s(G)$.  
  \end{theorem}

Any $\sigma\in S^m_{\rho,\delta}(G)$ admits a distributional  convolution kernel $\kappa: z\mapsto \kappa_x(z) \in \cC^\infty(G,\cS'(G))$, i.e. $\sigma(x,\pi) = \widehat \kappa_x(\pi)$ and 
$\Op^{\rm KN}(\sigma)f(x) =f*\kappa_x(x)$, $f\in \cS(G)$. 

\medskip 

In the next subsections, we  discuss the properties of the pseudodifferential calculus with respect to the application of a difference operator $\Delta_q$ for $q\in\mathcal S(G)$ and  with the convolution by a Schwartz function. Those properties were the main ingredients of the proof of G\aa rding inequality in the case of compact groups.

\subsubsection{Stability of the symbol classes with respect to difference operators}

The statement of Lemma~\ref{lem_contDelta_q} holds in the context of graded Lie groups if $a\in \cD(G)$ is replaced with $a\in \cS(G)$.

\begin{proof}[Sketch of the proof for Lemma \ref{lem_contDelta_q} for graded $G$ and $a$ Schwartz]
The properties of the  convolution kernels of symbols in $S^m_{\rho,\delta}(G)$, for instance being Schwartz away from the origin, implies that we may assume $q\in \cD(G)$ with compact support near the origin. Let $\chi\in \cD(G)$ be such that $\chi\equiv 1$ on the support of $q$.  
Let $P^q_N(z)$ be the Taylor polynomial at order $N$ in the sense of Folland-Stein \cite{folland+stein_82,FR}
for $q$.
The estimates of  the convolution kernel  $\kappa$ of $\sigma$ imply that
$$
 \sup_{x,y\in G}|\chi(q - P^q_N) \kappa_x|(y)
$$
 will be finite for $N$ large enough, with a constant given by semi-norms in $\sigma$, 
 and similarly for  $X_y^{\beta'_2}\tilde X_y^{\beta'_2}
  y^\alpha \chi(q - P^q_N) X^\beta_x\kappa_x $.
  This implies the statement. 
 \end{proof}

 \subsubsection{Stability of the symbol classes with respect to convolution}

 The analysis of the convolution of a symbol requires first to consider 
the properties  relatively with right-translation. The situation is  more involved because the group~$G$ is not compact. The analogue of Lemma~\ref{lem_translationSm} is the following.

\begin{lemma}
\label{lem_translationSm_graded}
Let $m\in \bR$ and $1\geq \rho\geq \delta\geq 0$.
 If $x_0\in G$, then for any $\sigma\in S_{\rho,\delta}^m(G)$, the symbol $R_{x_0}\sigma  = \{\sigma(xx_0,\pi) : (x,\pi)\in G\times \Ghat\}$ is in $S_{\rho,\delta}^m(G)$ and the map $(x_0,\sigma)\mapsto R_{x_0}\sigma$ is continuous  $G\times S_{\rho,\delta}^m(G)\to S_{\rho,\delta}^m(G)$.
Moreover, if we fix a homogeneous quasi-norm $|\cdot|$ on $G$, then for any semi-norm $\|\cdot\|_{S^m_{\rho,\delta}, a,b,c}$, there exists $N\in \N$ and $C>0$  such that
$$
\forall x_0\in G, \ 
\forall \sigma\in S^m_{\rho,\delta}(G),\qquad
 \| R_{x_0}\sigma \|_{S^{m}_{\rho,\delta},a,b,c,}\leq C(1+|x_0|)^N \| \sigma \| _{S^m_{\rho,\delta},a,b,c}.
$$
\end{lemma}

\begin{proof}
The proof follows the lines of the proof of Lemma~\ref{lem_translationSm} using \eqref{eq_Xalphafxx0}.
\end{proof}

We define the convolution of a symbol with a (suitable) function formally as in the compact case. As we will see below, 
certain properties of convoluting  a symbol will be more involved in the graded case because derivatives of  higher weights start appearing in the Taylor estimates on graded Lie groups, although on stratified Lie groups, only the derivatives of weight one occur
(see the proof of Lemma \ref{lem_convSm_graded} below). 
More precisely, the analogue of Lemma~\ref{lem_convSm} is the next
Lemma \ref{lem_translationSm_graded} and uses the notation:
\begin{equation}\label{def:upsilon}
 \upsilon := \left\{\begin{array}{ll}
 \upsilon_n & \mbox{if} \ G \ \mbox{is graded},\\
 1 & \mbox{if} \ G \ \mbox{is stratified}.\\
 \end{array}\right.
 \end{equation}

\begin{lemma}
\label{lem_convSm_graded}
Let $m\in \bR$ and $1\geq \rho\geq \delta\geq 0$.
\begin{enumerate}
 \item If $\sigma\in S_{\rho,\delta}^m(G)$ and $\varphi\in \cS(G)$, then 
	 we have $\sigma \,*\, \varphi \in S_{\rho,\delta}^{m}(G)$.
  Moreover,  for any semi-norm $\|\cdot\|_{S_{\rho,\delta}^m,a,b,c}  $ and $\varphi\in \cS(G)$, 
  there exists $C=C(\varphi,b)>0$ such that  
	 $$
  \forall \sigma\in S^m_{\rho,\delta}(G),\qquad
	\|\sigma \,*\, \varphi  \|_{S_{\rho,\delta}^m,a,b,c} 
	\leq C \|\sigma\|_{S_{\rho,\delta}^m,a,0,c}.
	$$
 This implies that $\sigma\mapsto \sigma*\varphi$ is continuous on  $S_{\rho,\delta}^m(G)$. 
 \item 
 Furthermore, if we fix a homogeneous quasi-norm $|\cdot|$ on $G$ and if $\int_G \varphi(y)dy =1$ then  for any semi-norm $\|\cdot\|_{S_{\rho,\delta}^{m+\upsilon \delta},a,b,c}  $, 
 there exists $N\in \bN_0$ and $C'=C'(b)>0$  such that for any $\sigma\in S^m_{\rho,\delta}(G)$, we have:
 $$
\|\sigma*\varphi  - \sigma\|
_{S^{m+\delta \upsilon  }_{\rho,\delta},a,b,c}
\leq 
C'  \int_G |y|^{\upsilon_1} (1+|y|)^N |\varphi(y)|dy \ 
\|\sigma\|_{S^{m}_{\rho,\delta},a,b+\upsilon,c}.$$
\end{enumerate}
\end{lemma}

Recall  that $\upsilon_1=1$ in the stratified case.

\begin{proof} Adapting the proof of the compact case, 
we observe that
$$
\Delta^\alpha X^\beta \sigma*\varphi(x,\pi)
= \int_G R_{y^{-1}}\Delta^\alpha \sigma(x,\pi)X^\beta \varphi(y) dy.
$$
Therefore,  $\varphi$ being Schwartz class and Lemma~\ref{lem_translationSm_graded}
together with \cite[Corollary 3.1.32]{FR} readily imply Part (1).

Assume now $\int_G \varphi(y)dy =1$. Then,
\[
(\sigma*\varphi  - \sigma)(x,\pi)= \int_G\left(\sigma(xy^{-1},\pi)-\sigma(x,\pi)\right) \varphi(y) dy,
\]
By  the Taylor estimates due to Folland and Stein \cite[Section 3.1.8]{FR}, 
we have
\begin{align*}
&  \|X^\beta_x(\sigma*\varphi  - \sigma)(x,\pi)\|_{\sL(\cH_\pi)}
\leq
\int_G 
\left\|X^\beta_x(\sigma(xy^{-1},\pi)-\sigma(x,\pi))\right\|_{\sL(\cH_\pi)} |\varphi(y)| dy
\\
&\qquad \lesssim 
\sum_{j=1}^n
  \int_G |y|^{\upsilon_j}
  \sup_{|y'|\lesssim |y|}
\left\|X_{j,y'}X^\beta_x\sigma(xy',\pi)\right\|_{\sL(\cH_\pi)} |\varphi(y)| dy
\\
&\qquad \lesssim 
\sum_{j=1}^n
  \int_G |y|^{\upsilon_j}
  (1+|y|)^N |\varphi(y)| dy \ 
  \max_{[\beta']=[\beta]+\upsilon_n}
  \sup_{x'\in G}
\left\|X^{\beta'}_x\sigma(x',\pi)\right\|_{\sL(\cH_\pi)} ,
\end{align*}
 for some $N\in \bN_0$,
by \eqref{eq_Xalphafxx0} and \cite[Corollary 3.1.32]{FR}.
This implies Part (2) when $m=0$, $a=0$, $\delta=0$ and $c=0$. The same arguments imply Part (2) for any semi-norm $\|\cdot\|_{S_{\rho,\delta}^{m+\upsilon_n \delta},a,b,c}  $.

In the stratified case, the Taylor estimates due  to Folland and Stein \cite[(1.41)]{folland+stein_82} involve only the left-invariant derivatives of weight $\upsilon_1=1$, yielding Part (2) in this case. 
\end{proof}

\subsection{Proof of the G\aa rding inequality}
Here, we prove the following $(\rho,\delta)$-generalisation of Theorem \ref{thm_Garding_graded}:
\begin{theorem}
\label{thm_Gardingnilp_rhodelta}
Let $G$ be a graded nilpotent Lie group. 
Let $m\in \bR$ and $1\geq \rho>\delta\geq 0$.
	Assume that the symbol $\sigma\in S_{\rho,\delta}^{m}(G)$ satisfies the positivity condition $\sigma \geq 0$.
	Then, for all $\eta>0$, there exists a constant $C_\eta>0$ such that 
	$$
	\forall f\in \cS(G),\qquad 
	\Re \left (\Op^{\rm KN}(\sigma) f, f\right )_{L^2(G)} \geq -\eta  \|f\|_{L^2_{\frac{m+\upsilon\delta}2}(G)}^2 -C_\eta  \|f\|_{L^2_{\frac{m-(\rho-\delta)}2}(G)}^2 
$$	 
where $\upsilon$ is defined in~\eqref{def:upsilon} (recall $\upsilon=1$ if $G$ is stratified).
\smallskip 

Moreover, if $\delta=0$  and  $\sigma\in S_{\rho,\delta}^{m}(G)$ satisfies the elliptic condition $\sigma_0 \geq c (\id  +\widehat \cR)^{\frac {m}\nu}$, 
		for some constant $c_0>0$ where $\cR$ is a positive Rockland operator of homogeneous degree $\nu$.
	Then there exist constants $c,C>0$ such that 
	$$
	\forall f\in \cS(G),\qquad 
	\Re \left (\Op^{\rm KN}(\sigma) f, f\right )_{L^2(G)} \geq c\|f\|_{L^2_{\frac{m}2}(G)}^2 -C \|f\|_{L^2_{\frac{m-\rho}2}(G)}^2 .
$$	 
\end{theorem}

The analogue of Remark~\ref{rem:pluto} is true  in the case of nilpotent Lie groups. 

The proof of  Theorem \ref{thm_Gardingnilp_rhodelta}  is an adaptation of the case of compact groups given in  Section~\ref{subsec_pfGarding_comp}. We first need to replace the Sobolev spaces $H^s(G)$ with the  Sobolev space $L^2_s(G)$ adapted to the graded nilpotent case. Moreover, in the final argument showing that  
the case of a symbol  $\sigma\in S_{\rho,\delta}^{m}(G)$  follows from the case of a symbol of order 0, we need to  replace $\sigma_1$ with 
$$
\sigma_1 = (\id+\widehat\cR)^{-\frac m {2\nu}} \sigma  
(\id+\widehat\cR)^{-\frac m {2\nu}}.
$$

Before detailing the proof for a symbol $\sigma\in S_{\rho,\delta}^0(G)$ of order 0, we discuss the two main ingredients: the use of an approximation of the identity and the comparison of the Kohn-Niremberg approximation with the Wick's one.  
We shall use the approximation of identity of Lemma~\ref{joe} and the following corollary of Lemma~\ref{lem_convSm_graded}:

\begin{corollary}
\label{cor_lem_convSm_graded}
We continue with the setting  of Lemma~\ref{joe}. 
Let
$m\in \bR$ and  $1\geq \rho> \delta\geq 0$.
For any semi-norm $\|\cdot \|_{S^{m+\upsilon \delta}_{\rho,\delta}(G),a_0,b_0,c_0}$ and any $\varphi_1\in \cS(G)$, 
there exists $C>0$  such that 
$$
\forall t\in (0,1],\qquad 
\forall \sigma\in S^m_{\rho,\delta}(G),\qquad 
\|\sigma -  \sigma*\varphi_t\|_{S^{m+\upsilon \delta}_{\rho,\delta}(G),a_0,b_0,c_0}\leq C \, t 
\| \sigma\|_{S^{m}_{\rho,\delta}(G),a_0,b_0+\upsilon,c_0},
$$
where $\nu$ is defined by \eqref{def:upsilon}.
\end{corollary}

\begin{proof}
    By Lemma~\ref{lem_convSm_graded} (2), 
\begin{align*}
    \|\sigma -  \sigma*\varphi_t\|_{S^{m+\upsilon \delta}_{\rho,\delta}(G),a_0,b_0,c_0}
    &\leq C' \| \sigma\|_{S^{m}_{\rho,\delta}(G),a_0,b_0+\upsilon,c_0}\int_G | y|^{\upsilon_1} (1+|y|)^N |\varphi_t(y)|dy\\
       &\leq C' \| \sigma\|_{S^{m}_{\rho,\delta}(G),a_0,b_0+\upsilon,c_0} \int_G | \delta_t y'|^{\upsilon_1} (1+|\delta_t y'|)^N |\varphi_1(y')|dy',
\end{align*}
after the change of variable $y=\delta_t y'$.
Since 
$$
| \delta_t y'|=t|y'|\qquad\mbox{and}\qquad
(1+|\delta_t y'|)\leq (1+t) (1+|y'|),
$$
the conclusion follows  for $t\in (0,1]$.
\end{proof}

As in compact groups, the main step in our proof will be the analysis of the Wick quantization in the $\Psi^\infty_{\rho,\delta}$-calculus. Indeed, if $\sigma\in S_{\rho,\delta}^m(G)$ with $m\leq 0$, then 
 $\sigma\in L^\infty(G\times \Ghat)$ and one can consider $\Op^{\rm Wick}(\sigma)$ and its membership in the $\Psi^\infty_{\rho,\delta}$-calculus.

\begin{lemma}
\label{lem_OpW_graded}
Let $a\in\cS(G)$ and 
 consider the Wick quantization $\Op^{{\rm Wick},a}$ with~$a$. 
Let $m\in \bR$ and $1\geq \rho\geq \delta\geq 0$ with $\delta\neq 1$.
\begin{enumerate}
	\item If $\sigma\in S_{\rho,\delta}^m(G)$ with $m\leq 0$, then 
 $\sigma\in L^\infty(G\times \Ghat)$ and $\Op^{\rm Wick}(\sigma)\in \Psi_{\rho,\delta}^m(G)$. Moreover, the map $\sigma \mapsto \Op^{\rm Wick}(\sigma)$ is continuous $S_{\rho,\delta}^m(G)\to \Psi_{\rho,\delta}^m(G)$.

 \item  If $\sigma\in S_{\rho,\delta}^m(G)$ with $m\leq 0$, 
	 then we have
	$$
	\Op^{\rm Wick}(\sigma) - \Op^{\rm KN} (\sigma *|a|^2) \in \Psi_{\rho,\delta}^{m-(\rho-\delta)}(G).
	$$ 
Moreover, the map $\sigma \mapsto \Op^{\rm Wick}(\sigma) - \Op^{\rm KN} (\sigma *|a|^2)$ is continuous $S_{\rho,\delta}^m(G)\to \Psi_{\rho,\delta}^{m-(\rho-\delta)}(G)$.
\end{enumerate}
\end{lemma}

\begin{proof}
We adapt the proof of Lemma \ref{lem_OpW} and start by  rephrasing Lemma \ref{lem_kW} as $\Op^{\rm Wick}(\sigma) = \Op^{\rm KN}(\sigma^{\rm Wick})$
with 
$$
\sigma^{\rm Wick} (x,\pi)= \int_G \Delta_{q_z} \sigma (xz^{-1},\pi) \ \bar a(z) \, dz
= \int_G R_{z^{-1}}\Delta_{q_z} \sigma (x,\pi)\ \bar a(z) \, dz
$$
where $q_z(w) := a(zw^{-1})$. 
Since Lemma~\ref{lem_contDelta_q} (1) also holds on graded nilpotent Lie groups $G$ for Schwartz functions, using 
Lemma \ref{lem_translationSm_graded} and  the fact that $a\in \cS(G)$,  we obtain Point (1).

For Point (2), we observe that 
$$
\int_G \Delta_{q_z(e_G)} \sigma (xz^{-1},\pi)\, \bar a(z)\,dz
=\int_G  |a|^{2}(z) \, \sigma(xz^{-1},\pi)\, dz
= \sigma*|a|^2 (x,\pi).
$$
Hence, 
\[
\sigma^{\rm Wick}(x,\pi)- \sigma*|a|^2 (x,\pi) =\int_G  R_{z^{-1} }\Delta_{q_z-q_z(e_G)}\sigma(x,\pi)\ \bar a(z) \,  dz,
\]
and 
we conclude using Lemma \ref{lem_contDelta_q} (2) for graded groups, the estimate of Lemma \ref{lem_translationSm_graded} and the fact that $a\in\cS(G)$.
\end{proof}

We can now prove  Theorem \ref{thm_Gardingnilp_rhodelta}.

\begin{proof}[Proof of Theorem~\ref{thm_Gardingnilp_rhodelta}]
As explained above, it suffices to show the statement for 
 $\sigma\in S_{\rho,\delta}^0(G)$ satisfying $\sigma(x,\pi)\geq 0$.
We fix a function $a_1\in \cS(G)$ with $\|a_1\|_{L^2(G)}=1$, 
and set $a_t(x):=t^{-Q/2}a_1(\delta_t^{-1} x)$, $x\in G$, $t>0$.
We observe that $|a_t|^2 = t^{-Q} |a_1|^2 \circ \delta_t^{-1}$, $t>0$, is an approximation of the identity in the sense of Lemma \ref{joe}. 

We now consider the Wick quantization $\Op^{{\rm Wick}, a}$ with $a=a_t$ to be chosen at the end of the proof.
The properties of the Wick quantization (see \eqref{positivity}) 
imply that for all $f\in \cS(G)$,
\begin{align*}
	\Re (\Op^{\rm KN}(\sigma) f, f)_{L^2(G)}
&\geq 
-\left| \left(\Op^{\rm KN} (\sigma -  \sigma*|a|^2 ) f, f\right)_{L^2(G)}\right|
 -  C
\|f\|^2_{L^2_{-\frac {\rho-\delta}2}(G)},
\end{align*}
for some constant $C>0$. 
We now choose $a=a_t$
with $t>0$ small enough so that, 
by Corollary \ref{cor_lem_convSm_graded},
$$
\left|\left(\Op^{\rm KN} (\sigma -  \sigma*|a|^2 ) f, f\right)_{L^2(G)}\right|
\leq \eta   \|f\|_{L^2_{\frac {\upsilon\delta} 2}(G)}^2.
$$
 This shows the case of $\sigma\in S_{\rho,\delta}^0(G)$ and  we can conclude the proof of Theorem \ref{thm_Gardingnilp_rhodelta} in a similar manner as for Theorem \ref{thm_Gardingcomp_rhodelta}. 
\end{proof}

\section{Semi-classical G\r  arding inequality on graded nilpotent Lie groups}
\label{sec_SC_nilpG}

In this section, we show the semi-classical inequality stated in Theorem \ref{thm:garding_sc}. The proof is inspired by Lemma~1.2 in~\cite{gerard_leichtnam}.
Before this, we recall the definition of the semi-classical calculus and we introduce the  Wick quantization adapted to the semi-classical setting.

\subsection{Semi-classical pseudodifferential calculus}
\label{subsec_semiclC}

The set $\mathcal A_0$ is the space of symbols
$\sigma = \{\sigma(x,\pi) : (x,\pi)\in G\times \widehat G\}$ of the form 
$$
\sigma(x,\pi)=\mathcal F \kappa_x (\pi) = \int_G \kappa_x(y) (\pi(y))^* dy, 
$$
where $(x,y)\mapsto \kappa_x(y)$ is a function of  the topological vector space  $\mathcal C_c^\infty(G,\mathcal S( G))$ of smooth and compactly supported functions in the variable $x \in G$ valued in the set of Schwartz class functions. 
As before, $x\mapsto \kappa_x$ is called the convolution kernel of $\sigma$. 

With the symbol $\sigma\in\mathcal A_0$, we associate the (family of) {\it semi-classical pseudodifferential operators} 
$$
{\rm Op}_\eps(\sigma) = \Op^{\rm KN} \left(\sigma (\cdot ,\delta_\eps \, \cdot)\right) , \qquad \eps\in (0,1], 
$$ 
where the Kohn-Nirenberg quantization $\Op^{\rm KN}$ was defined in Section \ref{subsubsec_OPKN}
and $\delta_r$ denotes the action of $\mathbb{R}^+$ 
on $\widehat G$ given via
$$
\delta_r \pi (x)
=
\pi(\delta_r x),\quad x\in G,\
\pi\in \widehat G, \ r>0.
$$
By Plancherel's theorem, in particular the uniqueness of the Plancherel measure $d\mu$, the latter is $Q$-homogeneous on $\widehat G$ for these dilations. 

In other words, we have
$$
{\rm Op}_\eps(\sigma) f(x) = \int_{\pi\in\widehat G} {\rm Tr}_{{\mathcal H}_\pi} \left(  \pi(x) \sigma(x,\delta_\eps\pi) {\mathcal F} f(\pi)  \right)d\mu(\pi),\;\;
f\in \mathcal S(G), \, x\in G.
$$
In terms of the convolution kernel $\kappa_x =\mathcal F^{-1} \sigma(x,\cdot)$, we have
$$
{\rm Op}_\eps(\sigma) f(x) = f* \kappa^{(\eps)}_x (x),\quad 
f\in \mathcal S(G), \, x\in G.
$$
Above, $\kappa^{(\eps)}_x$ is the 
convolution kernel of $\sigma (\cdot ,\delta_\eps \, \cdot) $ and is given by a rescaling of the convolution kernel of $\sigma$:
$$
\kappa^{(\eps)}_x (y) := \eps^{-Q} \kappa_x(\delta_\eps ^{-1} y).
$$

 \subsection{The semi-classical Wick quantization}\label{sec:frame}

Let $a\in \mathcal S(G)$ such that $\| a\|_{L^2(G)}=1$.
We set
$$
a_\eps := \eps^{-\frac Q4} a \circ \delta_{\eps^{-\frac 12}},
\qquad \eps>0,
$$
so that $a_\eps\in \cS(G)$ with 
 $\| a_\eps \|_{L^2(G)}=1$.
Moreover, for each $(x,\pi)\in G\times \Ghat$, we define the  operator on $\cH_\pi$ depending on $y\in G$,
$$
\sF_{x,\pi}^\eps (y)=  a_\eps (x^{-1} y)\delta_\eps ^{-1}\pi(y)^*,
$$
and define the operator $\cB^\eps$ on $\cS(G)$ via
$$
\cB^\eps [f](x,\pi)= \eps^{-\frac Q2}  \int_G f(y) \sF^\eps _{x,\pi}(y) dy,\quad f\in \cS(G), \ (x,\pi)\in G\times \Ghat.
$$
Note that, with respect to the operator $\cB_a$ defined in Section \ref{subsubsec_B}, we have
$$
\cB^\eps [f](x,\pi) =\eps^{-\frac Q2} \cB_{a_\eps}[f](x,\delta_{\eps}^{-1} \pi).
$$
Hence, by Proposition \ref{prop_B}, 
the map $\cB^\eps $ extends uniquely to an isometry from $L^2(G)$ to
$L^2(G\times \widehat G)$ for which we keep the same notation. Denoting by 
$\cB^{\eps,*} : L^2(G\times \widehat G)\to L^2(G)$  its adjoint map, we have $\cB^{\eps,*}\cB^\eps=\id_{L^2(G)}$ while $\cB^\eps \cB^{\eps,*}$ is a projection on a closed subspace of $L^2(G\times \Ghat)$.
\smallskip

Set for $(x,\pi)\in G\times \Ghat$,
\[ 
g_{x,\pi,k,\ell}^\ell(y):=  \left(\sF^\eps_{x,\pi}(y)^*\varphi_k(\pi),\varphi_\ell(\pi)\right)_{\cH_\pi},\;y\in G,
\]
where $(\cdot,\cdot)_{\cH_\pi}$ denotes the inner product of~$\mathcal H_\pi$ and $(\varphi_k(\pi))_{k\in I_\pi}$, $I_\pi\subset \N$ is an orthonormal basis of~$\mathcal H_\pi$. Then,  arguing as  in Corollary~\ref{cor:frame}, we obtain the semi-classical
 frame decomposition: 
    \[
f=\int_{G\times \widehat G}
\sum_{k,\ell\in I_\pi} \left( f, g_{x,\pi,k,\ell}^\eps\right)_{L^2(G)} \; g_{x,\pi,k,\ell}^\eps \; dxd\mu(\pi),\;\; f\in L^2(G).
    \]

\smallskip

We define the semi-classical Wick quantization for $\sigma\in L^\infty(\Ghat)$ 
$$
\Op^{\rm Wick}_\eps(\sigma) :=  \cB^{\eps,*} \sigma\, \cB^\eps.
$$
Here again, it is a positive quantization and 
we can compute the convolution kernel of $\Op^{\rm Wick}_\eps(\sigma)$ as in Lemma \ref{lem_kW}:

\begin{lemma}
If $\sigma\in \cA_0$, 
then 
$$
\Op^{\rm Wick}_\eps(\sigma) = \Op_\eps (\sigma^{\eps,{\rm Wick}}),
$$
where $\sigma^{\eps,{\rm Wick}}\in \cA_0$ has the convolution kernel 
\begin{align*}
	\kappa^{\eps,{\rm Wick}}_x(w) 
&=
\int_G a ( z'\delta_{\sqrt{\eps}} w^{-1})\bar a (z') \kappa_{x\delta_{\sqrt{\eps}}{z'}^{-1}}(w) dz'.
\end{align*}
\end{lemma}

\begin{proof}
Arguing as  in Lemma \ref{lem_kW}, we obtain (using changes of variables)
\begin{align*}
	\kappa^{\eps,{\rm Wick}}_x(w) 
&=
\int_G a_\eps (z^{-1}x\delta_\eps w^{-1})\bar a_\eps (z^{-1}x) \kappa_z(w) dz\\
&=
\int_G a_\eps (z'\delta_\eps w^{-1})\bar a_\eps (z') \kappa_{ x{z'}^{-1}}(w) dz'\\
&=
\int_G a ( z'\delta_{\sqrt{\eps}} w^{-1})\bar a (z') \kappa_{x\delta_{\sqrt{\eps}}{z'}^{-1}}(w) dz'.
\end{align*}
\end{proof}

\begin{corollary}
\label{cor_OpepsWKN}
	We choose a function $a\in \cD(G)$ that is even, i.e.  $a(x^{-1})=a(x)$, and real valued.  
	Then for any $\sigma\in \cA_0$, there exists $C>0$ such that for all $\eps\in (0,1]$, 
	$$
	\|\Op_\eps(\sigma)-\Op^{\rm Wick}_\eps (\sigma) \|_{\sL(L^2(G))} \leq C\eps.
	$$
\end{corollary}

\begin{proof}
By Lemma \ref{lem_A0norm}, using the $\cA_0$-norm defined in \eqref{eq_A0norm}, we have
$$
\|\Op_\eps(\sigma)-\Op^{\rm Wick}_\eps (\sigma) \|_{\sL(L^2(G))}
\leq \| \sigma - \sigma^{\eps,{\rm Wick}} \|_{\cA_0}
\leq I_1(\eps) +I_2(\eps), 
$$	
where
\begin{align*}
	I_1(\eps)&:= \int_G \sup_{x\in G}\left|\int_G |a(z)|^2 \left(\kappa_x(w) -\kappa_{x\delta_{\sqrt{\eps}}{z}^{-1}}(w) \right)dz \right| dw , \\
I_2(\eps)&:= \int_G \sup_{x\in G}\left|\int_G (a(z)-a(z\delta_{\sqrt \eps} w^{-1}))\bar a(z)  \kappa_{x\delta_{\sqrt{\eps}}{z}^{-1}}(w) dz \right| dw.
\end{align*}
By the Taylor estimates due to Folland and Stein \cite{folland+stein_82,FR}, if $\upsilon_1=1$, we have:
\begin{align*}
	I_1(\eps)&= \sqrt \eps   \int_G \sup_{x\in G}\left| 
\sum_{j=1}^{n_1} 
 \int_G (-z_j) |a(z)|^2 dz \ 
	    X_{j,x}\kappa_x(w)   \right|dw + O(\eps), \\
I_2(\eps)&=\sqrt \eps \int_G \sup_{x\in G}\left|
\sum_{j=1}^{n_1} 
\int_G   (- w_j ) \bar a(z)  X_j a(z)  \kappa_{x }(w) dz \right| dw +O(\eps)\\
&\leq \sqrt \eps 
\sum_{j=1}^{n_1}\left|
\int_G  
\bar a(z) X_ja (z)   dz \right|  \int_G |w_j| \sup_{x'\in G}\left| \kappa_{x'}(w)\right|  dw +O(\eps).
\end{align*}
We recall that $n_1$ denotes the dimension of the first strata (see paragraph~\ref{sec:adap_basis} where the  basis $(X_j)_{1\leq j\leq n}$ has been introduced). 
As $a$ is even, for any polynomial $q$ satisfying $q(z^{-1})=-q(z)$ such as the coordinate polynomials $z_j$,
we have
$\int_G |a(z)|^2 q(z) dz=0$.
As $a$ is real valued, for any left or right invariant vector field $X$, an integration by parts shows 
 $ \int_G X_j a(z) \bar a(z) dz =0$. 
Consequently, $I_1(\eps)=O(\eps)$ and $I_2(\eps)=O(\eps)$ if $\upsilon_1=1$. Moreover, if 
  $\upsilon_1>1$, then the Taylor estimate gives $I_1(\eps)+I_2(\eps)=O(\eps^{\frac{\upsilon_1}2})=O(\eps)$.
\end{proof}

\subsection{Proof of the semi-classical G\aa rding inequality}
Let $\sigma\in\mathcal A_0$ with $\sigma\geq 0$. 
By the properties of the semi-classical Wick quantisation, then 
$$
\left( \Op^{{\rm Wick}}_\eps(\sigma )f,f\right)_{L^2(G)} 
=(\sigma \cB^\eps f,\cB^\eps f)_{L^2(G\times \Ghat)} \geq 0.
$$
We write
\begin{align*}
\Re \left( \Op_\eps(\sigma )f,f\right)_{L^2(G)}
&\geq 
\left( \Op^{{\rm Wick}}_\eps(\sigma )f,f\right)_{L^2(G)}
- \|\Op_\eps(\sigma)-\Op^{\rm Wick}_\eps (\sigma) \|_{\sL(L^2(G))}\|f\|_{L^2(G)}^2\\
&\geq - \|\Op_\eps(\sigma)-\Op^{\rm Wick}_\eps (\sigma) \|_{\sL(L^2(G))}\|f\|_{L^2(G)}^2.
\end{align*}
By Corollary \ref{cor_OpepsWKN}, $\|\Op_\eps^{KN}(\sigma)-\Op^{\rm Wick}_\eps (\sigma) \|_{\sL(L^2(G))}=O(\eps)$.
This concludes the proof of Theorem~\ref{thm:garding_sc}.
\smallskip 

We point out that  the semi-classical case is more straightforward because we restrict ourselves to the use of $L^2$-norm and to a gain in the semi-classical parameter $\eps$. 
Moreover, we do not need a strong ellipticity assumption on the symbol $\sigma$ and its positivity is enough to conclude. This is  specific to the semi-classical setting.

\appendix

\section{The Euclidean case}
\label{app}

In this section, we recall the definitions and some properties of the 
Kohn-Nirenberg and Wick quantizations in the Euclidean case $\bR^n$. 
We develop the same chain of arguments  that show  the G\aa rding inequality  for the H\"ormander calculus on $\bR^n$  as in the core of the paper.
This leads  to a proof which is close to the one of  \cite[Chapter 2, section 6]{folland}, while leading to a weaker result.

\subsection{Kohn-Nirenberg and Wick quantizations}

On $\bR^n$, the Kohn-Nirenberg quantization may be defined for any symbols $\sigma\in \cS'(\bR^n\times \bR^n)$ via
the formula 
$$
\Op^{\rm KN} (\sigma)f (x)
=\int_{\bR^n} e^{2i\pi x\xi } \sigma(x,\xi)\, \widehat f (\xi) \, d\xi, \qquad x\in \bR^n,\, f\in \cS(\bR^n), 
$$
where $\widehat f = \cF f$ denotes the Euclidean Fourier transform of $f$: 
$$
\cF f(\xi)=
\widehat f(\xi) = \int_{\bR^n} e^{-2i\pi x\xi }f(x)dx, \qquad \xi\in \bR^n.
$$
With the \emph{convolution kernel} $\kappa_x := \cF^{-1}\sigma(x,\cdot\,)$ of $\sigma$, this may be rewritten as 
 $$
\Op^{\rm KN} (\sigma)f (x)
=f*\kappa_x(x), \qquad x\in \bR^n,\, f\in \cS(\bR^n).
$$

Fixing a continuous, bounded and square-integrable function $a$ with $\|a\|_{L^2(\bR^n)}=1$, we 
set for any $f\in L^2(\bR^n)$ and $(x,\xi)\in \bR^n$
$$
\cB_a [f] (x,\xi) := \cF (f \, a (\cdot -x)) (\xi)
= \int_{\bR^n} f(y)\, a(y-x) \, e^{-2i\pi y\xi } dy.
$$
This defines the generalised Bargmann transform $\cB_a = \cB$. The function $a$ is usually chosen as the Gaussian function $a(x)=\pi^{-\frac d4}{\rm e}^{-\frac{|x|^2}{2}} $ \cite{corobook}.  It is an isometry $L^2(\bR^n)\to L^2(\bR^n\times \bR^n)$.
Denoting by $\cB^*$ its adjoint, we define the Wick quantization $\Op^{\rm Wick} = \Op^{{\rm Wick},a}, $ for any symbol $\sigma\in L^\infty (\bR^n\times \bR^n)$ via:
$$
\Op^{\rm Wick} (\sigma) f= \cB^* \sigma \cB[f], 
\qquad f\in L^2(\bR^n).
$$
This quantization has the advantage of yielding bounded operators on $L^2$, of preserving self-adjointness:
$$
\|\Op^{\rm Wick} (\sigma)\|_{\sL(L^2(\bR^n)}
\leq \|\sigma\|_{L^\infty(\bR^n\times \bR^n)},
\qquad
\Op^{\rm Wick} (\sigma)^*=\Op^{\rm Wick} (\bar \sigma),
$$
and  positivity:
$$
\sigma(x,\xi)\geq 0 \ \mbox{for all}\ (x,\xi)\in \bR^n\times\bR^n \Longrightarrow
(\Op^{\rm Wick} (\sigma)f,f)_{L^2(\bR^n)}\geq 0.
$$

The link between the Kohn-Nirenberg and Wick quantization for a bounded symbol is the following:
\begin{lemma}
\label{lem_kappaWRn}
Let $a\in \cS(\bR^n)$ with $\|a\|_{L^2(\bR^n)}=1$, and consider the associated Wick quantization.
For any symbol $\sigma\in L^\infty(\bR^n\times \bR^n)$, we have: 
$$
\Op^{\rm Wick} (\sigma) f (x) = f*\kappa^{\rm Wick}_x(x),\quad f\in \cS(\bR^n), \ x\in \bR^n, 
$$	
where $\kappa^{\rm Wick}\in \cS'(\bR^n\times \bR^n)$ is given by 
$$
\kappa_x^{\rm Wick}(y) = \int_{\bR^n} a(z-y) \bar a(z) \kappa_{x-z}(y) dz,
$$
where $\kappa_x =\cF^{-1}\sigma(x,\cdot)$ denotes the convolution kernel of $\sigma$. 
Hence, $\Op^{\rm Wick}(\sigma) =\Op^{\rm KN}(\sigma^{\rm Wick})$ where $\sigma^{\rm Wick}\in \cS'(\bR^n\times \bR^n)$ is the symbol given by $\sigma^{\rm Wick}(x,\xi)=\cF \kappa^{\rm Wick}_x(\xi)$. 
\end{lemma}

\subsection{G\aa rding inequalities for H\"ormander symbols}
First, let us recall the definition of the H\"ormander classes of symbols.
\begin{definition}
	A function $\sigma\in C^\infty(\bR^n\times \bR^n)$ is a H\"ormander symbol of order $m\in \bR$ and index $(\rho,\delta)$
	when $$
\forall \alpha,\beta\in \bN_0^n, \quad\exists C_{\alpha,\beta}>0,
\qquad 
\forall (x,\xi)\in \bR^n\times\bR^n,\qquad 
|\partial_x^\beta\partial_\xi^\alpha \sigma(x,\xi) |
\leq 
C_{\alpha,\beta} (1+|\xi|^2)^{\frac {m-\rho|\alpha|+\delta|\beta|}2}.
$$
The space of H\"ormander symbols of order $m$ and index $(\rho,\delta)$ is denoted by $S^m_{\rho,\delta}(\bR^n)$.
\end{definition}
The space $S^m_{\rho,\delta}(\bR^n)$ is naturally equipped with a structure of Fr\'echet space, inherited by the resulting class of operators
 $\Psi_{\rho,\delta}^m(\bR^n):=\Op^{\rm KN}(S_{\rho,\delta}^m(\bR^n))$. Moreover, 
the H\"ormander calculus $\cup_{m\in \bR} \Psi_{\rho,\delta}^m(\bR^n)$ is a calculus in the sense of Definition \ref{def_pseudo-diff_calculus}. In this context, 
 the link between the Kohn-Nirenberg and Wick quantizations is given in Part (2) of the following statement.
 
\begin{proposition}
\label{prop_OpWRn}
Let $1\geq \rho\geq \delta\geq 0$ with $\delta\neq 1$.
\begin{enumerate}
\item Let $m\in \bR$. 
If $\varphi\in \cS(\bR^n)$ and  $\sigma\in S_{\rho,\delta}^m(\bR^n)$
then 
$$
\sigma*\varphi:(x,\xi) \longmapsto \int_{\bR^n} \varphi(z)\sigma(x-z,\xi) = \sigma*\varphi (\, \cdot\, ,\xi)\, (x), 
$$
defines a symbol in $S_{\rho,\delta}^m(\bR^n)$.
\item 
We assume that $a\in \cS(\bR^n)$ with $\|a\|_{L^2(\bR^n)}=1$. 
If  $\sigma\in S_{\rho,\delta}^0(\bR^n)$, then 
$$
	\Op^{\rm Wick}(\sigma) - \Op^{\rm KN}(\sigma*|a|^2) \in \Psi_{\rho,\delta}^{-(\rho-\delta)}(\bR^n).
	$$
 \end{enumerate}
\end{proposition}

\begin{proof}[Sketch of the proof of Proposition \ref{prop_OpWRn}]
Part (1) is easily checked. For Part (2),
we may rephrase Lemma \ref{lem_kappaWRn}  using the notion of difference operators $\Delta_q$ 
(which coincide with $\frac{i}{2\pi}\partial_{\xi_j}$ when $q=x_j$) 
defined formally via:
$$
(\Delta_q \sigma )(x,\xi) = \cF \left(q \cF^{-1} \sigma(x,\cdot)\right)(\xi) = \cF (q \kappa_x)(\xi) = \sigma(x ,\xi)*\widehat q, 
$$
with $\kappa_x =\cF^{-1}\sigma(x,\cdot)$ the convolution kernel of $\sigma$. We have:
$$
\Op^{\rm Wick}(\sigma) = \Op^{\rm KN}(\sigma^{\rm Wick}), 
\quad\mbox{with} \quad
\sigma^{\rm Wick} (x,\xi)= \int_{\bR^n} \Delta_{q_z} \sigma (x-z,\xi) dz,
$$
where $q_z(w) = a(z-w)\bar a(z)$.
We then conclude with the following asymptotic expansion in any symbol class $S^m_{\rho,\delta}(\bR^n)$  for a difference operator associated with $q\in \cS(\bR^n)$ 
$$
\Delta_q \sigma \ \sim \ q(0)  \sigma +  \sum_{\alpha>0} c_\alpha \partial^\alpha_x q(0) \partial_\xi^\alpha \sigma,
$$
with explicit coefficients $c_\alpha$.
\end{proof}

The properties  of the  Kohn-Nirenberg and Wick quantizations  imply the following G\aa rding inequality:
\begin{theorem}
\label{thm_gardingRn}
Let $m\in \bR$ and $1\geq \rho>\delta\geq 0$.
	Assume that the symbol $\sigma\in S_{\rho,\delta}^{m}(\bR^n)$ satisfies the elliptic condition $\sigma(x,\xi) \geq c (1  +|\xi|^2)^{m/2}$, 
	for some constant $c>0$.
	Then there exists a constant $C>0$ such that 
	\begin{equation}\label{eq:mickey}
	\forall f\in \cS(\bR^n),\qquad 
	\Re \left (\Op^{\rm KN}(\sigma) f, f\right )_{L^2(\bR^n)} \geq -C \|f\|_{H^{\frac{m-(\rho-\delta)}2}(\bR^n)}^2 .
\end{equation} 
\end{theorem}

Note that the so-called {\it sharp G\aa rding inequality} states that~\eqref{eq:mickey} holds under the weaker assumption $\sigma \geq 0$~\cite{folland}. The inequality of Theorem \ref{thm_gardingRn} is weaker, though useful and easy to prove.

\begin{proof}[Sketch of the proof of Theorem \ref{thm_gardingRn}]
It suffices to show the case $m=0$.
Let $\sigma\in S_{\rho,\delta}^0(\bR^n)$ satisfying $\sigma(x,\xi)\geq c$.
The properties  of the Wick quantization
(especially preserving positivity and Proposition \ref{prop_OpWRn}) imply
\begin{align}\label{averell}
	\Re (\Op^{\rm KN}(\sigma) f, f)_{L^2(\bR^n)}
&\geq  c \|f\|^2_{L^2(\bR^n)} 
- \|\Op^{\rm KN} (\sigma -  \sigma*|a|^2 )\|_{\sL(L^2(\bR^n))} \|f\|^2_{L^2(\bR^n)}
 -  C
\|f\|^2_{H^{\frac{\rho-\delta}2}(\bR^n)},
\end{align}
for some constant $C>0$. 
The operator $\Op^{\rm KN} (\sigma -  \sigma*|a|^2 )$ is bounded on $L^2(\bR^n)$ with operator norm estimated by a semi-norm in $\sigma - \sigma*|a|^2$. We may write this as:
$$
\|\Op^{\rm KN} (\sigma - \sigma*|a|^2)\|_{\sL(L^2(\bR^n))}
\leq C_1 \max_{\substack{|\alpha|\leq a_0\\ |\beta|\leq b_0}} \sup_{x,\xi\in \bR^n}
(1+|\xi|^2)^{-\frac{\delta|\beta|-\rho|\alpha|}2} 
\left|\partial_\xi^\alpha\partial_x^\beta
    (\sigma -  \sigma*|a|^2)(x,\xi)\right|,
$$
for some $a_0,b_0\in \bN$ and $C_1>0$.
We observe that the convolution above is in the variable $x$ only, so that denoting 
$\sigma_{\alpha,\beta}:= (1+|\xi|^2)^{-\frac{\delta|\beta|-\rho|\alpha|}2} 
\ \partial_\xi^\alpha\partial_x^\beta \sigma$, we have:
$$
(1+|\xi|^2)^{-\frac{\delta|\beta|-\rho|\alpha|}2} 
\ \partial_\xi^\alpha\partial_x^\beta
    (\sigma -  \sigma*|a|^2)
    =
\sigma_{\alpha,\beta} -  \sigma_{\alpha,\beta}*|a|^2 .
$$
Hence a judicious choice of $a$ in relation with an approximation of the identity will allow us to conclude. For this, we fix  a  function $a_1\in \cS(\bR^n)$ with $\|a_1\|_{L^2(\bR^n)}=1$
and set $a_t(x) :=t^{-n/2} a_1(t^{-1} x)$. 
We observe that $|a_t|^2 = t^{-n}|a_1 (t^{-1}\cdot)|^2$, $t>0$
 is an approximation of the identity.
We then choose $a=a_t$ with $t>0$ small enough so that  the right-hand side in~\eqref{averell} is $\leq c$.
\end{proof}

 \end{document}